\documentclass[a4paper,10pt]{article}
 
\usepackage{ae,lmodern} 
\usepackage[utf8]{inputenc}
\usepackage[T1]{fontenc}
\usepackage[english]{babel}
\usepackage{geometry}
\usepackage{url,xspace}
\usepackage{amsmath,amssymb,stmaryrd,spalign,bbm,mathtools,mathrsfs,dsfont,thmtools}
\usepackage[thmmarks,amsmath,framed,thref,hyperref]{ntheorem}
\usepackage{comment}
\usepackage[colorlinks=true,citecolor=teal,linkcolor=blue]{hyperref}
\usepackage{array}
\usepackage{enumitem}
\usepackage{changepage}
\usepackage{euscript}
\usepackage{cite}
\usepackage{graphicx}
\usepackage{pgfplots}
\usepackage[toc,page]{appendix}
\usepackage{etoolbox}

\counterwithin{equation}{section}

\setlist[enumerate,1]{label={(\roman*)}}

\pgfplotsset{width=10cm,compat=1.9}

\newcommand{\zerodisplayskips}{%
  \setlength{\abovedisplayskip}{5pt}%
  \setlength{\belowdisplayskip}{5pt}%
  \setlength{\abovedisplayshortskip}{5pt}%
  \setlength{\belowdisplayshortskip}{5pt}}
\appto{\normalsize}{\zerodisplayskips}
\appto{\small}{\zerodisplayskips}
\appto{\footnotesize}{\zerodisplayskips}

\newtheorem{theorem}{Theorem}[section]
    \newtheorem{corollary}[theorem]{Corollary}
    \newtheorem{lemma}[theorem]{Lemma}
    \newtheorem{proposition}[theorem]{Proposition}

    \newtheorem{assumption}[theorem]{Assumption}
    
    \theorembodyfont{\upshape}
    \newtheorem{definition}[theorem]{Definition}
    
    \newtheorem{remark}[theorem]{Remark}
    
\counterwithin{equation}{section}

\setlist[enumerate,1]{label={(\roman*)}}

\theoremstyle{nonumberplain}
\theoremheaderfont{\itshape}
\theorembodyfont{\normalfont}
\theoremseparator{.\,—}
\theoremsymbol{}
\newtheorem{proof-wo}{Proof}
\theoremsymbol{\ensuremath{\blacksquare}}
\newtheorem{proof}{Proof}

\newcommand{\eus}{\EuScript}

\newcommand{\llb}{\llbracket}

\title{Homogeneous Sobolev global-in-time maximal regularity and related trace estimates }
\author{Anatole \textsc{Gaudin}\thanks{{Aix-Marseille Université, CNRS, Centrale Marseille, I2M, Marseille, France} - \textbf{email:} anatole.gaudin@univ-amu.fr}}

\begin{document}

\maketitle
\begin{abstract} In this paper, we prove global-in-time $\dot{\mathrm{H}}^{\alpha,q}$-maximal regularity for a class of injective, but not invertible, sectorial operators on a UMD Banach space $X$, provided $q\in(1,+\infty)$, $\alpha\in(-1+1/q,1/q)$. We also prove the corresponding trace estimate, so that the solution to the canonical abstract Cauchy problem is continuous with values in a not necessarily complete trace space.

In order to put our result in perspective, we also provide a short review on $\mathrm{L}^q$-maximal regularity which includes some recent advances such as the revisited homogeneous operator and interpolation theory by Danchin, Hieber, Mucha and Tolksdorf. This theory will be used to build the appropriate trace space, from which we want to choose the initial data, and the solution of our abstract Cauchy problem to fall in. 
\end{abstract}

\addtocontents{toc}{\protect\thispagestyle{empty}}
\tableofcontents

\renewcommand{\arraystretch}{1.5}



\section{Introduction}

\subsection{Motivations and interests}

\subsubsection{The example of the Laplacian on \texorpdfstring{$\mathbb{R}^n$}{Rn}}
The  $\mathrm{L}^q$-maximal regularity is very powerful and fundamental tool for the study of a wide range parabolic partial differential equations, that comes mainly from physics, geometry or chemistry.

The usual theory built for sectorial operators on UMD Banach have been widely investigated. However, when it comes to look where the solution of the abstract Cauchy problem lies as a continuous function of time, we have to restrict ourselves either to control in finite time, or to ask the operator to be invertible.

For instance, let us take a look at the scalar heat equation on $\mathbb{R}^n$, for $T\in(0,+\infty]$, $u_0\in\eus{S}'(\mathbb{R}^n)$, $f\in\mathrm{L}^1_{\text{loc}}([0,T],\eus{S}'(\mathbb{R}^n))$,
\begin{align}\tag{HE}\label{HE}
    \left\{\begin{array}{rl}
            \partial_t u(t) -\Delta u(t)  =& f(t) \,\text{, } 0<t<T \text{, }\\
            u(0) =& u_0\text{. }
    \end{array}
    \right.\text{. }
\end{align}
The standard theory, see e.g. \cite[Remark~4.10.9]{AmmanBookParabolicVolI1995}, \cite[Theorem~3.5.5]{PrussSimonett2016}, tells us that provided $f\in \mathrm{L}^q((0,T),\mathrm{L}^p(\mathbb{R}^n))$, $u_0\in \mathrm{B}^{2-2/q}_{p,q}(\mathbb{R}^n)$, $T<+\infty$, $p,q\in(1,+\infty)$, the Cauchy problem \eqref{HE} admits a unique solution $u\in \mathrm{H}^{1,q}((0,T),\mathrm{L}^p(\mathbb{R}^n))\cap \mathrm{L}^{q}((0,T),\mathrm{H}^{2,p}(\mathbb{R}^n))\subset \mathrm{C}^0([0,T],\mathrm{B}^{2-2/q}_{p,q}(\mathbb{R}^n))$
which satisfies the estimates
\begin{align*}
    \lVert u\rVert_{\mathrm{L}^{\infty}([0,T],\mathrm{B}^{2-2/q}_{p,q})}\lesssim_{n,p,q,T} \lVert(u , \partial_t u, \Delta u)\rVert_{\mathrm{L}^{q}((0,T),\mathrm{L}^{p})} \lesssim_{n,p,q,T}  \lVert f \rVert_{\mathrm{L}^{q}((0,T),\mathrm{L}^{p})} + \lVert u_0\rVert_{\mathrm{B}^{2-2/q}_{p,q}} \text{.}
\end{align*}
In these estimates, implicit constants are dependent of $T$ and blow up as $T$ goes to infinity. It is in fact even worse than that : one cannot expect a global-in-time estimate of this type. Indeed, such a control on the term $\lVert u \rVert_{\mathrm{L}^{q}([0,+\infty),\mathrm{L}^{p})}$ would imply that the Laplacian $\Delta$ is invertible on $\mathrm{L}^p(\mathbb{R}^n)$, see for instance \cite[Section~2]{CoulhonLamberton1986} or  \cite[Corollary~3.5.3]{PrussSimonett2016}, which is known to be false.

However, when $f\in \mathrm{L}^q(\mathbb{R}_+,\mathrm{L}^p(\mathbb{R}^n))$, $u_0\in \dot{\mathrm{B}}^{2-2/q}_{p,q}(\mathbb{R}^n)$, there is still a unique solution $u$ to \eqref{HE} such that $\partial_t u,\Delta u \in \mathrm{L}^q(\mathbb{R}_+,\mathrm{L}^p(\mathbb{R}^n))$ and $u\in \mathrm{C}^0_b(\mathbb{R}_+,\dot{\mathrm{B}}^{2-2/q}_{p,q}(\mathbb{R}^n))$ with the global-in-time estimate
\begin{align}\label{eq:GiTLqLpMRestHeatEq}
    \lVert u\rVert_{\mathrm{L}^{\infty}(\mathbb{R}_+,\dot{\mathrm{B}}^{2-2/q}_{p,q})}\lesssim_{n,p,q} \lVert(\partial_t u, \Delta u)\rVert_{\mathrm{L}^{q}(\mathbb{R}_+,\mathrm{L}^{p})} \lesssim_{n,p,q}  \lVert f \rVert_{\mathrm{L}^{q}(\mathbb{R}_+,\mathrm{L}^{p})} + \lVert u_0\rVert_{\dot{\mathrm{B}}^{2-2/q}_{p,q}} \text{.}
\end{align}
This result is well known, but while the right hand side estimate of \eqref{eq:GiTLqLpMRestHeatEq} arises from the usual theory when $u_0=0$, see e.g. \cite[Proposition~8.3.4,~Corollary~9.3.12]{bookHaase2006}, this is however not the case for the left hand side \textbf{trace estimate} and to obtain the space from which we choose the initial data $u_0$, see for instance \cite[Theorem~2.34]{bookBahouriCheminDanchin}. A reason is that the usual theory for traces in maximal regularity will only produce an inhomogeneous Besov space, which is not suitable: it makes us lose again the uniform control with respect to time on the left hand side part of \eqref{eq:GiTLqLpMRestHeatEq}. The same kind of issue would happen for other injective, but non-invertible, sectorial operators. We further hope we have convinced the reader that the general theory cannot be applied for global-in-time estimate for the very well-known Cauchy problem \eqref{HE} which is a sufficiently important issue.

\subsubsection{On the choice of function spaces.}

When it comes to the study of actual partial differential equations, it would be interesting to play with integrability, decay-in-time, or even with some Sobolev regularity in-time of possible solutions for the linear part of the problem. A wide development of the theory of power-weighted fractional Sobolev-in-time maximal regularity is made and applied, and can be found in \cite[Sections~3.2,~3.4~\&~3.5]{PrussSimonett2016}. Pr\"{u}ss and Simonett gave the complete construction of maximal regularity results for spaces of the type $\mathrm{H}^{\alpha,q}_{\mu,0}(\mathbb{R}_+,X)$, $\alpha \in [0,1]$, $\mu\in(1/q,1]$, which stands for the space of measurable functions $u$ such that
\begin{align*}
    t\mapsto t^{1-\mu}u(t) \in \mathrm{H}^{\alpha,q}_{1,0}(\mathbb{R}_+,X)\text{.}
\end{align*}
Here, $\mathrm{H}^{\alpha,q}_{1,0}$ coincides with the standard Sobolev space with zero boundary condition. The applications to general quasilinear parabolic partial differential equations of the $\mathrm{L}^q_{\mu}$-maximal regularity have also been extensively reviewed in \cite{KhonePrussWilke2010,LeCronePrussWilke2010}.

We also mention \cite{Pruss2002} which contains a treatment of fractional Sobolev in-time maximal regularity theory as well as a review of mixed derivative estimates. It was recently used, for instance, in \cite{BrandoleseMonniaux2021} for the study of the Boussinesq system where mixed derivative estimates were their main tool in combination with the usual $\mathrm{L}^q$-maximal regularity setting. 

However, both of previous treatments do not allow, again, global-in-time estimates for injective, but not invertible, operators such as the Laplacian on $\mathbb{R}^n$. This is where our idea  comes from in order to keep the possibility of playing the Sobolev in-time regularity: we want to show global-in-time homogeneous \textbf{$\dot{\mathrm{H}}^{\alpha,q}$-maximal regularity} for non-zero initial data, with a trace estimate similar to the one in \eqref{eq:GiTLqLpMRestHeatEq}, $q\in(1,+\infty)$, $\alpha\in(-1+1/q,1/q)$.

Danchin, Hieber, Mucha and Tolksdorf, in \cite[Chapter~2]{DanchinHieberMuchaTolk2020}, provide global-in-time estimates for injective, but not invertible, operators in the framework of the Da Prato-Grisvard $\mathrm{L}^q$-maximal regularity. Let us motivate here their idea of homogeneous interpolation and operator theory for injective sectorial operators from an other point of view. Indeed, in the previous example for the heat equation, if we set $X_p=\mathrm{L}^p$ $A=-\Delta$, $\mathrm{D}_p(A)=\mathrm{H}^{2,p}$, the Besov space used as trace space is given by the real interpolation space
\begin{align*}
    \mathrm{B}^{2-2/q}_{p,q} = (X_p, \mathrm{D}_p(A))_{1-\frac{1}{q},q}\text{,}
\end{align*}
and this follows from the general trace theory. See e.g. \cite[Section~1.2]{bookLunardiInterpTheory}, \cite[Section~3.4]{PrussSimonett2016} or \cite[Section~4]{MeyriesVeraar2014} for even fancier and more general function spaces.

Our idea is to say that the homogeneous Besov space yielding the homogeneous estimate \eqref{eq:GiTLqLpMRestHeatEq} would be given by
\begin{align*}
    \dot{\mathrm{B}}^{2-2/q}_{p,q} = (X_p, \mathrm{D}_p(\mathring{A}))_{1-\frac{1}{q},q}\text{,}
\end{align*}
where, here, $\mathrm{D}_p(\mathring{A})=\dot{\mathrm{H}}^{2,p}$ which is also, at least morally, the closure of $\mathrm{D}_p({A})$ under the (semi-) norm $\lVert A \cdot\rVert_{X_p}\sim_{p,n}\lVert \nabla^2 \cdot\rVert_{\mathrm{L}^p}$. And this is exactly the kind of construction achieved in \cite[Chapter~2]{DanchinHieberMuchaTolk2020} for abstract sectorial operators, in order to obtain a global-in-time Da Prato-Grisvard $\mathrm{L}^q$-maximal regularity theorem. Moreover, such a construction avoids the need of completeness for $\mathrm{D}(\mathring{A})$ which is fundamental in the scope of the treatment of some non-linear partial differential equations with global-in-time estimates. Indeed, realization of homogeneous functions spaces that are usually employed are not necessarily complete on their whole scale, see for instance \cite{bookBahouriCheminDanchin,DanchinMucha2009,DanchinMucha2015,DanchinHieberMuchaTolk2020,Gaudin2022} and the references therein. We notice that the possible lack of completeness of $\mathrm{D}(\mathring{A})$ implies that the resulting real interpolation space $(X, \mathrm{D}(\mathring{A}))_{\theta,q}$ is not necessarily complete either, but this is somewhat mandatory to deal with actual non-linear or boundary value problems.

This is very important to emphasize the fact that in concrete applications involving non-linearities or boundary values, one cannot avoid the completeness issue taking the completion instead. For instance, if we work with the completion of $\dot{\mathrm{H}}^{2,p}$, $p\geqslant n/2$, one might end up with elements that are no longer even distributions which is really inconvenient. If instead one chooses the realization up to polynomials, then one ends up with products rules that depends on the choice of representation: it is not clear how to choose properly the polynomial part in a canonical way. Moreover, point-wise composition with a global diffeomorphism, or bi-Lipschitz map, as done in \cite[Chapters~5~\&~7]{DanchinHieberMuchaTolk2020}, would be meaningless if one works with the setting "up to polynomials".

Those issues concerning the completion also prevent the use of \textit{standard} homogeneous operator and interpolation theory started in \cite[Chapter~6,~Sections~6.3~\&~6.4]{bookHaase2006}, then extended in \cite{HaakHaaseKunstmann2006}, requiring in the end to work with $\mathrm{D}(\mathring{A})$ as a complete space.

We notice that the recent work \cite{AgrestiLindemulderVeraar2022} does not apply in our setting to obtain the desired trace estimate. There are two reasons: $\mathrm{D}(\mathring{A})$ is not an actual completion, and their work do not take in consideration homogeneous fractional Sobolev scale for the time variable.

\subsection{Notations, definitions}

For $X$ a Banach space, $(\Omega,\mu)$ a sigma finite measure space, and $p\in[1,+\infty]$, $\mathrm{L}^p(\Omega,\mu,X)$ stands for the space of (Bochner-)measurable functions $u\,:\,\Omega\longrightarrow X$, such that $t\mapsto\lVert u(t)\rVert_X \in \mathrm{L}^p(\Omega,\mu,\mathbb{R})$.

A Banach space $X$ is said to have the \textbf{Unconditionnal Martingale Difference} property (or to be \textbf{UMD}) if the Hilbert transform is bounded on $\mathrm{L}^q(\mathbb{R},X)$ for one (or equivalently all) $q\in(1,+\infty)$.

For two real numbers $A,B\in\mathbb{R}$, $A\lesssim_{a,b,c} B$ means that there exists a constant $C>0$ depending on ${a,b,c}$ such that $A\leqslant C B$. When $A\lesssim_{a,b,c} B$ and $B \lesssim_{a,b,c} A$ are true, we simply write $A\sim_{a,b,c} B$.

\subsubsection{Sectorial operators on Banach spaces}

We introduce the following subsets of the complex plane
\begin{align*}
    \Sigma_\mu &:=\{ \,z\in\mathbb{C}^\ast\,:\,\lvert\mathrm{arg}(z)\rvert<\mu\,\}\text{, if } \mu\in(0,\pi)\text{, }
\end{align*}
we also define $\Sigma_0 := (0,+\infty)$, and we are going to consider the closure $\overline{\Sigma}_\mu$.

An operator $(\mathrm{D}(A),A)$ on complex valued Banach space $X$ is said to be $\omega$-\textit{\textbf{sectorial}}, if for a fixed $\omega\in (0,\pi)$, both conditions are satisfied \textit{
\begin{enumerate}
    \item  \textit{$\sigma(A)\subset \overline{\Sigma}_\omega $, where $\sigma(A)$ stands for the spectrum of $A$ };
    \item \textit{for all $\mu\in(\omega,\pi)$, $\sup_{\lambda\in \mathbb{C}\setminus\overline{\Sigma}_\mu}\lVert \lambda(\lambda \mathrm{I}-A)^{-1}\rVert_{X\rightarrow X} < +\infty$.}
\end{enumerate}
}

For $(\mathrm{D}(A),A)$ injective and $\omega$-sectorial with $\omega\in[0,\pi)$, we say that $A$ has  {\textbf{bounded imaginary powers}} (BIP) of type $\theta_A\geqslant 0$ if for all $x\in \mathrm{D}(A)\cap\mathrm{R}(A)$, for $f(z)=z^{is}$, 
\begin{align*}\label{eq:DunfordintegralFuncCalc}
    f(A)x := \frac{1}{2 i \pi}\int_{\partial \Sigma_\theta} f(z)(z\mathrm{I}-A)^{-1}x\,\mathrm{d}z\text{, }
\end{align*}
for some $\theta \in(\omega,\pi)$, with $\partial \Sigma_\theta$ oriented counterclockwise, yields a bounded linear operator for all $s\in\mathbb{R}$, and
\begin{align*}
    \theta_A := \inf \left\{ \nu\geqslant 0\,\bigg|\, \sup_{s\in\mathbb{R}} e^{-\nu|s|}\lVert A^{is}\rVert_{X \rightarrow X}<+\infty\right\}\text{ . }
\end{align*}

The functional calculus of sectorial operators is widely reviewed in several references but we mention here Haase's book \cite{bookHaase2006}. For a treatment of operator theory in the scope of $\mathrm{L}^q$-maximal regularity, such as BIP, we refer to \cite[Chapters~1~\&~2]{bookDenkHieberPruss2003} and \cite[Chapters~3~\&~4]{PrussSimonett2016}.

\subsection{Road map}

In Section~3: we provide a short construction of the homogeneous Sobolev spaces we need.
In order to achieve this, we will need to assume that the Banach space $X$ is still have the UMD property. This is to ensure that we have a suitable definition of $\dot{\mathrm{H}}^{\alpha,q}(\mathbb{R}_+,X)$, since we will need some complex interpolation theory requiring bounded imaginary powers for the time derivative, see e.g. \cite[Theorems~6.7~\&~6.8]{LindemulderMeyriesVeerar2018}. 

Before that, in Section~2, we give a review of the current state of standard $\mathrm{L}^q$-maximal regularity with global-in-time estimates: the treatment will be made first on UMD Banach spaces $X$. A second part is dedicated to a review of the homogeneous operator and interpolation theory revisited by Danchin, Hieber, Mucha and Tolksdorf, with its application to Da Prato-Girsvard $\mathrm{L}^q$-maximal regularity.

Section~4 is devoted to our main result about $\dot{\mathrm{H}}^{\alpha,q}$-maximal regularity for some injective sectorial operators, with trace estimate in the possibly non-complete space $(X, \mathrm{D}(\mathring{A}))_{1+\alpha-\frac{1}{q},q}$. Before proving the main result Theorem \ref{thm:LqMaxRegUMDHomogeneous}, one has to prove that the quantities involved to solve the Cauchy problem are in fact well-defined, which is the goal of the preceding subparts.

\subsubsection*{Acknowledgment} The author would like to thank Sylvie Monniaux and Pascal Auscher for their useful remarks during earlier presentations of the current work. The author would also like to thank Bernhard H. Haak for pointing out the article \cite{HaakHaaseKunstmann2006}.


\section{Short state of the art for \texorpdfstring{$\mathrm{L}^q$}{Lq}-maximal regularity}\label{MaxRegReviewIntro}

We are going to recall here few facts about $\mathrm{L}^q$-maximal regularity ($q\in(1,+\infty)$) on UMD Banach spaces. We will also deal with the $\mathrm{L}^q$-maximal regularity provided by the Da Prato-Grisvard theory in both versions: inhomogeneous and homogeneous, both allowing under appropriate circumstance $q=1,+\infty$, allowing also to get rid of the UMD property on $X$.

\subsection{Review for the usual \texorpdfstring{$\mathrm{L}^q$}{Lq}-maximal regularity}

First, let us consider $(\mathrm{D}(A),A)$ a densely defined closed operator on a Banach space $X$. It is known, see \cite[Theorem~3.7.11]{ArendtBattyHieberNeubranker2011}, that the two following assertions are equivalent:\textit{
\begin{enumerate}
    \item $A$ is $\omega$-sectorial on $X$, with $\omega\in [0,\tfrac{\pi}{2})$;
    \item $-A$ generates a bounded holomorphic $\mathrm{C}_0$-semigroup on $X$, denoted by $(e^{-tA})_{t\geqslant0}$.
\end{enumerate}}

Thus, provided that $A$ is $\omega$-sectorial on $X$ for some $\omega\in [0,\tfrac{\pi}{2})$, for $T\in(0,+\infty]$, we look at the following abstract Cauchy problem, 
\begin{align}\tag{ACP}\label{ACP}
    \left\{\begin{array}{rl}
            \partial_t u(t) +Au(t)  =& f(t) \,\text{, } 0<t<T \text{, }\\
            u(0) =& u_0\text{. }
    \end{array}
    \right.\text{, }
\end{align}
where $f\in \mathrm{L}^1_{\mathrm{loc}}([0,T),X)$, $u_0 \in Y$, $Y$ being some normed vector space depending on $X$ and $\mathrm{D}(A)$.

And it turns out, see \cite[Proposition~3.1.16]{ArendtBattyHieberNeubranker2011}, that in our case for $u_0\in X$, $f\in \mathrm{L}^1((0,T),X)$, integral solutions $u\in \mathrm{C}^0([0,T],X)$ for \eqref{ACP} is unique, also called the \textbf{mild solution} of \eqref{ACP} and given by
\begin{align*}
    u(t)= e^{-tA}u_0 + \int_{0}^{t} e^{-(t-s)A}f(s)\,{\mathrm{d}s} \text{, } \quad 0\leqslant t<T\text{.}
\end{align*}

The question is: for a given $q\in [1,+\infty]$, can we find an appropriate space $Y$ (depending on $X$, $\mathrm{D}(A)$ and possibly $q$), such that if $u_0\in Y$ and $f\in \mathrm{L}^q((0,T),X)$, then \eqref{ACP} admits a unique solution $u$, satisfying $\partial_t u$, $Au\in\mathrm{L}^q((0,T),X)$, with norm control
\begin{align*}
    \lVert (\partial_t u, Au)\rVert_{\mathrm{L}^q((0,T),X)} \lesssim_{q,A} \lVert f\rVert_{\mathrm{L}^q((0,T),X)} + \lVert u_0\rVert_{Y} \text{ ? } 
\end{align*}

The problem \eqref{ACP} being linear, we introduce two related subproblems:
\begin{itemize}
    \item (ACP${}^0$) stands for \eqref{ACP} with $f=0$,
    \item (ACP${}_0$) stands for \eqref{ACP} with $u_0=0$,
\end{itemize}
recalling that according to basic $\mathrm{C}_0$-semigroup theory, $u=0$ is the unique solution of (ACP${}^0_0$). Hence, if \eqref{ACP} admits a solution, such solution is unique due to linearity so that it suffices to treat separately both problem (ACP${}^0$) and (ACP${}_0$).

$\bullet$ \, \textbf{For the }(ACP${}^0$)\textbf{ problem}, we introduce two quantities for $v\in X + \mathrm{D}(A)$,

\begin{align*}
    \lVert v\rVert_{\mathring{\eus{D}}_{A}(\theta,q)} := \left( \int_{0}^{+\infty} (t^{1-\theta}\lVert Ae^{-tA} v \rVert_{X})^q \frac{\mathrm{d}t}{t}\right)^\frac{1}{q}\text{, and }  \lVert v\rVert_{{\eus{D}}_{A}(\theta,q)} := \lVert v \rVert_X + \lVert v\rVert_{\mathring{\eus{D}}_{A}(\theta,q)}\text{, }
\end{align*}
where $\theta \in(0,1)$, $q\in[1,+\infty]$. This leads to the construction of the vector space
\begin{align*}
    {\eus{D}}_{A}(\theta,q) := \{ v\in X\,|\, \lVert v\rVert_{\mathring{\eus{D}}_{A}(\theta,q)}<+\infty \} \text{.}
\end{align*}
The vector space ${\eus{D}}_{A}(\theta,q)$ is known to be a Banach space under the norm $\lVert \cdot\rVert_{{\eus{D}}_{A}(\theta,q)}$ and moreover it satisfies the following equality with equivalence of norms
\begin{align}\label{eq:InterpolationXDomainAIntro}
    {\eus{D}}_{A}(\theta,q) = (X,\mathrm{D}(A))_{\theta,q}\text{, }
\end{align}
see \cite[Theorem~6.2.9]{bookHaase2006}. If moreover, $0\in \rho(A)$ it has been proved, \cite[Corollary~6.5.5]{bookHaase2006}, that $\lVert \cdot\rVert_{\mathring{\eus{D}}_{A}(\theta,q)}$ and $\lVert \cdot\rVert_{{\eus{D}}_{A}(\theta,q)}$ are two equivalent norms on ${\eus{D}}_{A}(\theta,q)$.

By definition, for all $u_0\in{\eus{D}}_{A}(1-{1}/{q},q)$, for $t\mapsto u(t)=e^{-tA}u_0$ the solution of (ACP${}^0$), we have
\begin{align*}
    \lVert u \rVert_{\mathrm{L}^{\infty}(\mathbb{R}_+,\mathring{\eus{D}}_{A}(1-{1}/{q},q))} \lesssim_{q,A} \lVert \partial_t u\rVert_{\mathrm{L}^q(\mathbb{R}_+,X)}=\lVert Au\rVert_{\mathrm{L}^q(\mathbb{R}_+,X)} = \lVert u_0\rVert_{\mathring{\eus{D}}_{A}(1-{1}/{q},q)}\text{, }
\end{align*}
and we also have, for all $T>0$,
\begin{align*}
    \lVert u \rVert_{\mathrm{L}^{\infty}(\mathbb{R}_+,{\eus{D}}_{A}(1-{1}/{q},q))} \lesssim_{q,A} \lVert u_0\rVert_{{\eus{D}}_{A}(1-{1}/{q},q)}\text{ and  }  \lVert u \rVert_{\mathrm{L}^q((0,T),X)} \lesssim_{q,A} T^{\tfrac{1}{q}} \lVert u_0\rVert_X \text{. }
\end{align*}
If moreover, $0\in \rho(A)$, then
\begin{align*}
    \lVert u \rVert_{\mathrm{L}^q(\mathbb{R}_+,X)} \lesssim_{q,A} \lVert u_0\rVert_X \text{. }
\end{align*}

$\bullet$ \, \textbf{For the }(ACP${}_0$)\textbf{ problem}, the question is much more delicate. In fact, the solution $u$ to (ACP${}_0$) is formally given by the Duhamel formula
\begin{align}\label{eqmaxregsolutionACP0}
    u(t) = \int_{0}^{t} e^{-(t-s)A}f(s)\,\mathrm{d}s\text{, } t>0 \text{, }  
\end{align}
and since, $\partial_t u = -Au +f$, it suffices to know whether
\begin{align}\label{eqmaxregestimateACP0}
    \lVert  Au\rVert_{\mathrm{L}^q(\mathbb{R}_+,X)} \lesssim_{q,A} \lVert f\rVert_{\mathrm{L}^q(\mathbb{R}_+,X)} \text{.}
\end{align}

This leads to the following definition:
\begin{definition}\label{def:LqMaxReg} The operator $A$ is said to have the \textbf{$\mathrm{\mathbf{L}}^{q}$-maximal regularity} property on $X$ if the solution $u$ given by \eqref{eqmaxregsolutionACP0} satisfies the above estimate \eqref{eqmaxregestimateACP0}.
\end{definition}

Let us remark that the case of finite time $T>0$ with the corresponding estimate can be easily deduced by \eqref{eqmaxregestimateACP0} applied to $\Tilde{f}$, the extension of $f$ to $\mathbb{R}_+$ by $0$, and the uniqueness of (ACP${}_0$).

It has been proved by Coulhon and Lamberton \cite{CoulhonLamberton1986}, that the property of the $\mathrm{L}^q$-maximal regularity does not depends on $q\in(1,\infty)$. See also \cite{deSimon64} for the first version of this result in the hilbertian-valued case.

Coulhon and Lamberton also showed, see \cite[Theorem~5.1]{CoulhonLamberton1986}, that the UMD property is a necessary condition for the Poisson semigroup to have the $\mathrm{L}^q$-maximal regularity property. The canonical example, provided $p\in(1,+\infty)$, is that $X=\mathrm{L}^p(\Omega)$ is a UMD space and so are its closed subspaces, see for instance  \cite[Propositions~4.2.15~\&~4.2.17]{HytonenNeervenVeraarWeisbookVolI2016}.

The following fact proved by Kalton and Lancien \cite{KaltonLancien2000}: for each non-hilbertian Banach lattice, there exists a sectorial operator such that \eqref{eqmaxregestimateACP0} fails.

However, for UMD Banach spaces, a full and definitive characterization of operators that satisfy $\mathrm{L}^q$-maximal regularity property has been proved by Weis {\cite[Theorem~4.2]{Weis2001}}. One may also check \cite[Theorem~1.11]{KunstmannWeis2004}, \cite[Theorem~4.4]{bookDenkHieberPruss2003} for other proofs and more details about $\mathcal{R}$-boundedness and its equivalence with $\mathrm{L}^q$-maximal regularity for sectorial operators on a UMD Banach space.

In practice, we rather use other results such has the Dore-Venni Theorem, \cite[Theorem~2.1]{DoreVenni1987}, which asserts that the boundedness of imaginary powers of $A$ with type $\theta_A<\frac{\pi}{2}$ is a sufficient condition to recover $\mathrm{L}^q$-maximal regularity for $q\in(1,+\infty)$. We mention \cite[Corollary~9.3.12]{bookHaase2006} for the same result that does not require invertibility of $A$.
In particular, the bounded holomorphic functional calculus of $A$ is a sufficient condition to recover $\mathrm{L}^q$-maximal regularity with $q\in(1,+\infty)$.

We may combine all results for (ACP${}^0$) and (ACP${}_0$) to state the following well-known $\mathrm{L}^q$-maximal regularity theorem, where we only state it with the sufficient condition of BIP for convenience.

\begin{theorem}\label{thm:LqMaxRegUMD}Let $\omega\in [0,\frac{\pi}{2})$, $(\mathrm{D}(A),A)$ an $\omega$-sectorial operator on a UMD Banach space $X$. Assume that $A$ has BIP of type $\theta_A<\frac{\pi}{2}$.

Let $q\in(1,+\infty)$ and $T\in(0,+\infty]$. For $f\in\mathrm{L}^q((0,T),X)$, $u_0\in \eus{D}_{A}(1-{1}/{q},q)$, the problem \eqref{ACP} admits a unique solution $u$ such that $\partial_t u$, $Au \in \mathrm{L}^q((0,T),X)$ with estimate
\begin{align}\label{BoundLqMaxReg}
    \lVert (\partial_t u, Au)\rVert_{\mathrm{L}^q((0,T),X)} \lesssim_{A,q} \lVert f\rVert_{\mathrm{L}^q((0,T),X)} + \lVert u_0\rVert_{\mathring{\eus{D}}_{A}(1-{1}/{q},q)}\text{. }
\end{align}
In addition, for all $T_\ast\leqslant T$, $T_\ast<+\infty$, we have $u\in \mathrm{C}^0([0,T_\ast),{\eus{D}}_{A}(1-{1}/{q},q))\cap \mathrm{L}^q((0,T_\ast),X)$ with estimates
\begin{align}
    \lVert u \rVert_{\mathrm{L}^\infty([0,T_\ast],{\eus{D}}_{A}(1-{1}/{q},q))} &\lesssim_{A,q,T_\ast} \lVert f\rVert_{\mathrm{L}^q((0,T_\ast),X)} + \lVert u_0\rVert_{{\eus{D}}_{A}(1-{1}/{q},q)}\text{, } \label{ContinuityLqMaxReg} \\
    \lVert u\rVert_{\mathrm{L}^q((0,T_\ast),X)} &\lesssim_{A,q} (T_\ast\lVert f\rVert_{\mathrm{L}^q((0,T_\ast),X)} + T_\ast^{\frac{1}{q}}\lVert u_0\rVert_{X})\text{. } \label{LqTBoundLqMaxReg}
\end{align}
If moreover $0\in\rho(A)$, we have
\begin{align}
    \lVert u\rVert_{\mathrm{L}^q((0,T),X)} &\lesssim_{A,q} \lVert f\rVert_{\mathrm{L}^q((0,T),X)} + \lVert u_0\rVert_{X} \text{. }\label{LqBoundLqMaxReg}
\end{align}
so that \eqref{ContinuityLqMaxReg} holds with uniform constant with respect to $T_\ast$, hence remains true for $T_\ast=+\infty$.
\end{theorem}

We comment the appearance of \eqref{ContinuityLqMaxReg} : it is a consequence of the trace theory for initial data in $\mathrm{L}^q$-maximal regularity which is itself a consequence of interpolation theory, see \cite[Chapter~4, Theorem~4.10.2]{AmmanBookParabolicVolI1995}, see also \cite[Corollary~1.14]{bookLunardiInterpTheory}. The appearance of \eqref{LqBoundLqMaxReg} comes from invertibility of $A$, so that it suffices to apply \eqref{eqmaxregestimateACP0}.

However, the approach used to obtain Theorem \ref{thm:LqMaxRegUMD} prevents $\mathrm{L}^1$ and $\mathrm{L}^\infty$-maximal regularity on $X$. Indeed, the UMD property requires the space $X$ to be at least reflexive, which is not the case for all spaces that are of use in partial differential equations (one may think about endpoint Besov spaces like $\mathrm{B}^{s}_{p,1}$ and $\mathrm{B}^{s}_{p,\infty}$, or even the space of continuous bounded functions $\mathrm{C}^0_b$).

\subsection{Revisited homogeneous operator and interpolation theory and global-in-time estimate for the Da Prato-Grisvard \texorpdfstring{$\mathrm{L}^q$}{Lq}-maximal regularity}

To overcome such difficulties, we present a theorem due to Da Prato and Grisvard \cite{DaPratoGrisvard1975}, where the idea was to replace $X$ by $\eus{D}_{A}(\theta,q)$, and look for $\mathrm{L}^q$-maximal regularity property on it instead of $X$, allowing $q=1$.

\begin{theorem}[ {\cite[Theorem~4.15]{DaPratoGrisvard1975}} ]\label{thm:DaPratoGrisvard1975} Let $\omega\in [0,\frac{\pi}{2})$, $(\mathrm{D}(A),A)$ an $\omega$-sectorial operator on a Banach space $X$. Let $q\in [1,+\infty)$, $\theta\in(0,\tfrac{1}{q})$, $\theta_q := \theta +1-1/q$, and let $T\in (0,+\infty)$.

For $f\in \mathrm{L}^q((0,T),{\eus{D}}_{A}(\theta,q))$ and $u_0\in {\eus{D}}_{A}(\theta_q,q)$, the problem \eqref{ACP} admits a unique mild solution
\begin{align*}
    u\in \mathrm{C}^0_b([0,T],{\eus{D}}_{A}(\theta_q,q))\text{,}
\end{align*}
such that $\partial_t u$, $Au\in \mathrm{L}^q((0,T),{\eus{D}}_{A}(\theta,q))$ with estimate
\begin{align}
    \lVert u\rVert_{\mathrm{L}^\infty([0,T],{\eus{D}}_{A}(\theta_q,q))} \lesssim_{A,\theta,q,T} \lVert (\partial_t u, Au)\rVert_{\mathrm{L}^q((0,T),{\eus{D}}_{A}(\theta,q))} \lesssim_{A,\theta,q,T} \lVert f\rVert_{\mathrm{L}^q((0,T),{\eus{D}}_{A}(\theta,q))} + \lVert u_0\rVert_{{\eus{D}}_{A}(\theta_q,q)}\text{. } \label{estimateLqMaxRegDaPratoGrisvard}
\end{align}
If moreover $0\in\rho(A)$, \eqref{estimateLqMaxRegDaPratoGrisvard} still holds with uniform constant with respect to $T$, allowing $T=+\infty$.
\end{theorem}
This Da Prato-Grisvard theorem does not have global in time estimate if $0\notin \rho (A)$, as was the case for the estimate \eqref{BoundLqMaxReg} of Theorem \ref{thm:LqMaxRegUMD}. The estimate \eqref{BoundLqMaxReg} is uniform in time : this is due to the fact that the estimate is \textbf{homogeneous}. This keypoint was captured in the work of Danchin, Hieber, Mucha and Tolksdorf \cite[Chapter~2]{DanchinHieberMuchaTolk2020} to build an \textbf{homogeneous} version of the Da~Prato-Grivard theorem for injective sectorial operators under some additional assumptions on $A$. We are going to present briefly their construction.
\begin{assumption}\label{asmpt:homogeneousdomaindef} The operator $(\mathrm{D}(A),A)$ is injective on $X$, and there exists a normed vector space $(Y, \left\lVert \cdot \right\rVert_{Y})$, such that for all $x\in \mathrm{D}(A)$,
\begin{align}
        \left\lVert Ax\right\rVert_{X} \sim \left\lVert x \right\rVert_{Y} \text{. }
\end{align}
\end{assumption}

The idea is to construct an homogeneous version of $A$ denoted $\mathring{A}$, defining first its domain
\begin{align*}
    \mathrm{D}(\mathring{A}) := \{\, y\in Y\, |\, \exists (x_n)_{n\in\mathbb{N}}\subset \mathrm{D}(A),\, \left\lVert y-x_n\right\rVert_Y\underset{n\rightarrow +\infty}{\longrightarrow} 0   \,\}\text{. }
\end{align*}
Then, for all $y\in \mathrm{D}(\mathring{A})$,
\begin{align*}
    \mathring{A}y := \lim_{n\rightarrow +\infty} Ax_n \text{. }
\end{align*}
Constructed this way, the operator $\mathring{A}$ is then injective on $\mathrm{D}(\mathring{A})$. We notice that $\mathrm{D}(\mathring{A})$ is a normed vector space, but not necessarily complete. We also need the existence of a Hausdorff topological vector space $Z$, such that $X,Y\subset Z$, and to consider the following assumption 
\begin{assumption}\label{asmpt:homogeneousdomainintersect} The operator $(\mathrm{D}(A),A)$ and the normed vector space $Y$ are such that
\begin{align}
        X\cap\mathrm{D}(\mathring{A}) = \mathrm{D}({A}) \text{. }
\end{align}
\end{assumption}
As a consequence of all above assumptions, we can extend naturally, see \cite[Remark~2.7]{DanchinHieberMuchaTolk2020}, $(e^{-tA})_{t\geqslant0}$ to a $\mathrm{C}_0$-semigroup,
\begin{align*}
        e^{-tA}\,:\, X+\mathrm{D}(\mathring{A}) \longrightarrow X+\mathrm{D}(\mathring{A})\text{ , } t\geqslant 0 \text{, }
\end{align*}
by the mean of the following formula for all $(x_0,a_0)\in X\times \mathrm{D}(\mathring{A})$, $t\geqslant 0$,
\begin{align}\label{eq:DefExtendedsemigroup}
    e^{-tA}(x_0+a_0) := e^{-tA}x_0 + \left(a_0-\int_{0}^{t} e^{-\tau A}\mathring{A} a_0 \,\mathrm{d} \tau\right).
\end{align}
and so that for $u_0\in X + \mathrm{D}(\mathring{A})$, and fixed $t$, the value above does not depend on the choice of decomposition $u_0=x_0+a_0$, see \cite[Proposition~2.6]{DanchinHieberMuchaTolk2020}.

Moreover, for all $u_0=x_0+a_0\in X + \mathrm{D}(\mathring{A})$, it is straight forward to see from \eqref{eq:DefExtendedsemigroup} and \cite[Proposition~2.6]{DanchinHieberMuchaTolk2020}, that $t\mapsto e^{-tA}u_0$ is strongly differentiable at any order with continuous derivatives on $(0,+\infty)$ taking its values in $X$. For $k\in\llb 1,+\infty\llb$, $t>0$, by analyticity of the semigroup
\begin{align*}
    (-\partial_t)^k(e^{-(\cdot)A}u_0)(t) = A^{k}e^{-tA}x_0 + A^{k-1}e^{-tA}\mathring{A}a_0 =A^{k-1}\mathring{A}e^{-tA}u_0 \in \mathrm{D}(A)\subset X\text{.}
\end{align*}

From there, one can fully make sense of the following vector space,
\begin{align*}
    \mathring{\eus{D}}_{A}(\theta,q) := \left\{ v\in X+\mathrm{D}(\mathring{A})\,\big{|}\, \lVert v\rVert_{\mathring{\eus{D}}_{A}(\theta,q)}<+\infty \right\} \text{.}
\end{align*}
Similarly to what happens for ${\eus{D}}_{A}(\theta,q)$ in \eqref{eq:InterpolationXDomainAIntro}, it has been proved in \cite[Proposition~2.12]{DanchinHieberMuchaTolk2020}, that the following equality holds with equivalence of norms,
\begin{align}
    \mathring{\eus{D}}_{A}(\theta,q) = (X,\mathrm{D}(\mathring{A}))_{\theta,q}\text{. }
\end{align}
However, the lack of completeness for $\mathrm{D}(\mathring{A})$ implies that $\mathring{\eus{D}}_{A}(\theta,q)$ is not necessarily complete. This has consequences on how to consider the forcing term $f$ in \eqref{ACP}, choosing $f\in \mathrm{L}^q((0,T),{\eus{D}}_{A}(\theta,q))$ instead of $f\in \mathrm{L}^q((0,T),\mathring{\eus{D}}_{A}(\theta,q))$ to avoid definition issues, the latter choice being possible when $\mathring{\eus{D}}_{A}(\theta,q)$ is a Banach space.

\begin{theorem}{\textbf{\textit{(\cite[Theorem~2.20]{DanchinHieberMuchaTolk2020})}}} \label{thm:DaPratoGrisvardHom2020} Let $\omega\in [0,\frac{\pi}{2})$, $(\mathrm{D}(A),A)$ an $\omega$-sectorial operator on a Banach space $X$ such that Assumptions \eqref{asmpt:homogeneousdomaindef} and \eqref{asmpt:homogeneousdomainintersect} are satisfied. Let $q\in [1,+\infty)$, $\theta\in(0,\tfrac{1}{q})$, $\theta_q := \theta +1-1/q$, and let $T\in(0,+\infty]$.

For $f\in \mathrm{L}^q((0,T),{\eus{D}}_{A}(\theta,q))$ and $u_0\in \mathring{\eus{D}}_{A}(\theta_q,q)$, the problem \eqref{ACP} admits a unique mild solution
\begin{align*}
    u\in \mathrm{C}^0_{b}([0,T],\mathring{\eus{D}}_{A}(\theta_q,q))\text{,}
\end{align*}
such that $\partial_t u$, $Au\in \mathrm{L}^q((0,T),\mathring{\eus{D}}_{A}(\theta,q))$ with estimates,
\begin{align}
    \lVert u\rVert_{\mathrm{L}^\infty([0,T],\mathring{\eus{D}}_{A}(\theta_q,q))} + \lVert (\partial_t u, Au)\rVert_{\mathrm{L}^q((0,T),\mathring{\eus{D}}_{A}(\theta,q))} \lesssim_{q,A} \lVert f\rVert_{\mathrm{L}^q((0,T),\mathring{\eus{D}}_{A}(\theta,q))} + \lVert u_0\rVert_{\mathring{\eus{D}}_{A}(\theta_q,q)}\text{. } \label{estimateLqMaxRegDaPratoGrisvardHom}
\end{align}
In case $q=+\infty$, we assume in addition that $u_0\in \mathrm{D}(A^2)$ and then for each $\theta\in (0,1)$,
\begin{align}
    \lVert (\partial_t u, Au)\rVert_{\mathrm{L}^\infty([0,T],\mathring{\eus{D}}_{A}(\theta,\infty))} \lesssim_{q,A} \lVert f\rVert_{\mathrm{L}^\infty((0,T),{\eus{D}}_{A}(\theta,\infty))} + \lVert A u_0\rVert_{\mathring{\eus{D}}_{A}(\theta,\infty)}\text{. }
\end{align}
\end{theorem}


\section{Vector-valued Sobolev spaces in a UMD Banach space and the time derivative}

\subsection{Banach valued Bessel and Riesz potential Sobolev spaces}

This subsection is devoted to few reminders on Bessel potential spaces on the whole line with values in a Banach space $X$ which is known to be UMD. This will be based on the constructions provided  by \cite{MeyriesVeraar2012,LindemulderMeyriesVeerar2018} and \cite[Chapter~5,~Section~5.6]{HytonenNeervenVeraarWeisbookVolI2016}, see also the references therein. 
From the properties we are going to gather about Bessel potential Sobolev spaces, we will be able give a simple construction of homogeneous (Riesz potential) Sobolev spaces with values in $X$, for regularity index near $0$. Namely, we will focus on the regularity index $\alpha\in(-1+1/q,1/q)$ when $q\in(1,+\infty)$.

We chose to investigate the Sobolev space on the one the half-line for which we won't have additional compatibility conditions at $0$. Notice that the condition on the regularity index is also here in order to avoid troubles of definition. Indeed,  the definition of homogeneous function spaces for regularity exponents beyond $1/q$ is not clear and a choice of realization have to be done, even in the scalar case. Such choice implies generally the loss of one, or more, usual and useful properties, like either the loss of distribution theory, the loss of completeness on the whole scale, or the loss of pointwise/meaningful (para-)products (in the scalar case, $X=\mathbb{C}$). See for instance \cite{bookBahouriCheminDanchin,DanchinHieberMuchaTolk2020,Gaudin2022,bookSawano2018} and reference therein for various constructions and addressed issues in the scalar-valued case.

From there and until the end of this article, we assume that $X$ has the UMD property. We recall that such space $X$ is necessarily reflexive.

\begin{definition}For $q\in(1,+\infty)$, $\alpha\in\mathbb{R}$, we define the vector space
\begin{align*}
    \mathrm{H}^{\alpha,q}(\mathbb{R},X) := \left\{\,u\in\eus{S}'(\mathbb{R},X)\,\big|\, (\mathrm{I}-\partial_{x}^2)^{\frac{\alpha}{2}}u\in\mathrm{L}^{q}(\mathbb{R},X)\,\right\}
\end{align*}
with its associated norm
\begin{align*}
    \lVert u\rVert_{\mathrm{H}^{\alpha,q}(\mathbb{R},X)} := \lVert (\mathrm{I}-\partial_{x}^2)^{\frac{\alpha}{2}}u\rVert_{\mathrm{L}^{q}(\mathbb{R},X)}\text{.}
\end{align*}
Here, $(\mathrm{I}-\partial_{x}^2)^{\frac{\alpha}{2}}$ have to be understood as the usual Fourier multiplier operator.
\end{definition}

Before going further, we introduce
\begin{align*}
    \eus{S}_0(\mathbb{R},X):= \left\{\,u\in\eus{S}(\mathbb{R},X)\,\big|\, \mathrm{supp}(\eus{F}u) \text{ is compact, } 0\notin \mathrm{supp}(\eus{F}u)\,\right\}\text{.}
\end{align*}

\begin{proposition}\label{prop:PropertiesinhomSobolevSpaces}Let $q\in(1,+\infty)$, $\alpha\in\mathbb{R}$, the following properties are true :
\begin{enumerate}
    \item $\mathrm{H}^{\alpha,q}(\mathbb{R},X)$ is a reflexive Banach space with
    \begin{align*}
        (\mathrm{H}^{\alpha,q}(\mathbb{R},X))^\ast = \mathrm{H}^{-\alpha,q'}(\mathbb{R},X^\ast)\text{;}
    \end{align*}
    \item $\eus{S}_0(\mathbb{R},X)$ is a dense subspace of $\mathrm{H}^{\alpha,q}(\mathbb{R},X)$;
    \item Provided $\alpha\geqslant 0$, for all $u\in\eus{S}'(\mathbb{R},X)$,
    \begin{align*}\label{eq:EquivofNormsHomInhomHaq}
        \lVert u\rVert_{\mathrm{H}^{\alpha,q}(\mathbb{R},X)} \sim_{\alpha,q,X} \lVert u\rVert_{\mathrm{L}^{q}(\mathbb{R},X)} + \lVert (-\partial_{x}^2)^{\frac{\alpha}{2}}u\rVert_{\mathrm{L}^{q}(\mathbb{R},X)}\text{;}
    \end{align*}
    \item Provided $\alpha \in [0,1/q)$, $\frac{1}{r}=\frac{1}{q}-\alpha$, for all $u\in \mathrm{H}^{\alpha,q}(\mathbb{R},X)$,
    \begin{align*}
        \lVert u\rVert_{\mathrm{L}^{r}(\mathbb{R},X)} \lesssim_{\alpha,q,X} \lVert (-\partial_{x}^2)^{\frac{\alpha}{2}}u\rVert_{\mathrm{L}^{q}(\mathbb{R},X)}\text{;}
    \end{align*}
    \item Provided $\alpha \in (-1+1/q,1/q)$, for all $u\in \mathrm{H}^{\alpha,q}(\mathbb{R},X)$,
    \begin{align*}
        \lVert \mathbbm{1}_{\mathbb{R}_+}u\rVert_{\mathrm{H}^{\alpha,q}(\mathbb{R},X)} \lesssim_{\alpha,q,X} \lVert u\rVert_{\mathrm{H}^{\alpha,q}(\mathbb{R},X)}\text{.}
    \end{align*}
\end{enumerate}
\end{proposition}

\begin{proof} Point \textit{(i)} is standard. Point \textit{(ii)} is a direct consequence of the corresponding results for $\alpha=0$, see \cite[Lemma~E.5.2]{bookHaase2006}. Point \textit{(iii)} is just \cite[Lemma~4.2]{LindemulderMeyriesVeerar2018}. Point \textit{(iv)} follows from \cite[Corollary~1.4]{MeyriesVeraar2012}, point \textit{(iii)} and a dilation argument. Point \textit{(v)} is just \cite[Theorem~4.1]{LindemulderMeyriesVeerar2018}.
\end{proof}

For $q\in(1,+\infty)$, $\alpha\in(-\infty,1/q)$, thanks to the points \textit{(i)}, \textit{(ii)} and \textit{(iv)} from the Proposition~\ref{prop:PropertiesinhomSobolevSpaces}, we introduce the quantity
\begin{align*}
    \lVert u\rVert_{\dot{\mathrm{H}}^{\alpha,q}(\mathbb{R},X)}:=\lVert (-\partial_{x}^2)^{\frac{\alpha}{2}}u\rVert_{\mathrm{L}^{q}(\mathbb{R},X)}\text{,}
\end{align*}
one can consider the completion
\begin{align*}
    \dot{\mathrm{H}}^{\alpha,q}(\mathbb{R},X) := \overline{\eus{S}_0(\mathbb{R},X)}^{\lVert \cdot\rVert_{\dot{\mathrm{H}}^{\alpha,q}(\mathbb{R},X)}}
\end{align*}
so that the next definition is meaningful.

\begin{definition}For $q\in(1,+\infty)$, $\alpha<1/q$, we define the vector spaces\textit{
\begin{enumerate}
    \item for $\alpha \geqslant 0$, $\frac{1}{r}:=\frac{1}{q}-\alpha$,
    \begin{align*}
    \dot{\mathrm{H}}^{\alpha,q}(\mathbb{R},X) := \left\{\,u\in\mathrm{L}^{r}(\mathbb{R},X)\,\big|\, (-\partial_{x}^2)^{\frac{\alpha}{2}}u\in\mathrm{L}^{q}(\mathbb{R},X)\,\right\}\text{,}
    \end{align*}
    \item for $\alpha \leqslant 0$,
    \begin{align*}
    \dot{\mathrm{H}}^{\alpha,q}(\mathbb{R},X) := \left\{\,u\in\mathrm{H}^{\alpha,q}(\mathbb{R},X)\,\big|\, (-\partial_{x}^2)^{\frac{\alpha}{2}}u\in\mathrm{L}^{q}(\mathbb{R},X)\,\right\}\text{,}
    \end{align*}
\end{enumerate}}
with their associated norm
\begin{align*}
    \lVert u\rVert_{\dot{\mathrm{H}}^{\alpha,q}(\mathbb{R},X)} := \lVert (-\partial_{x}^2)^{\frac{\alpha}{2}}u\rVert_{\mathrm{L}^{q}(\mathbb{R},X)}\text{.}
\end{align*}
Here, $(-\partial_{x}^2)^{\frac{\alpha}{2}}$ have to be understood as the usual Fourier multiplier operator.
\end{definition}

In a similar way, we obtain the following collection of properties.

\begin{proposition}\label{prop:PropertiesHomSobolevSpaces}Let $q\in(1,+\infty)$, $\alpha\in<1/q$, the following properties are true :
\begin{enumerate}
    \item for $\beta\in\mathbb{R}$, such that $\alpha+\beta<1/q$,
    \begin{align*}
        (-\partial_{x}^2)^{\frac{\beta}{2}}\,:\,\dot{\mathrm{H}}^{\alpha+\beta,q}(\mathbb{R},X)\longrightarrow \dot{\mathrm{H}}^{\alpha,q}(\mathbb{R},X)
    \end{align*}
    is an isomorphism of Banach spaces;
    \item $\dot{\mathrm{H}}^{\alpha,q}(\mathbb{R},X)$ is reflexive, and whenever $\alpha\in(-1+1/q,1/q)$,
    \begin{align*}
        (\dot{\mathrm{H}}^{\alpha,q}(\mathbb{R},X))^\ast = \dot{\mathrm{H}}^{-\alpha,q'}(\mathbb{R},X^\ast)\text{;}
    \end{align*}
    \item $\eus{S}_0(\mathbb{R},X)$ is a dense subspace of $\dot{\mathrm{H}}^{\alpha,q}(\mathbb{R},X)$;
    \item provided $\alpha \in [0,1/q)$, $\frac{1}{r}=\frac{1}{q}-\alpha$, for all $u\in \dot{\mathrm{H}}^{\alpha,q}(\mathbb{R},X)$, $v\in\mathrm{L}^{r'}(\mathbb{R},X)$,
    \begin{align*}
        &\lVert u\rVert_{\mathrm{L}^{r}(\mathbb{R},X)} \lesssim_{\alpha,q,X} \lVert u\rVert_{\dot{\mathrm{H}}^{\alpha,q}(\mathbb{R},X)}\text{,}\\
        &\lVert v\rVert_{\dot{\mathrm{H}}^{-\alpha,q'}(\mathbb{R},X)} \lesssim_{\alpha,q,X} \lVert v\rVert_{\mathrm{L}^{r'}(\mathbb{R},X)}\text{;}
    \end{align*}
    \item provided $\alpha \in (-1+1/q,1/q)$, for all $u\in \dot{\mathrm{H}}^{\alpha,q}(\mathbb{R},X)$,
    \begin{align*}
        \lVert \mathbbm{1}_{\mathbb{R}_+}u\rVert_{\dot{\mathrm{H}}^{\alpha,q}(\mathbb{R},X)} \lesssim_{\alpha,q,X} \lVert u\rVert_{\dot{\mathrm{H}}^{\alpha,q}(\mathbb{R},X)}\text{.}
    \end{align*}
\end{enumerate}
\end{proposition}

\begin{proof} Point \textit{(iii)} follows from the definition. Point \textit{(i)} is straightforward by density of $\eus{S}_0(\mathbb{R},X)$. Point \textit{(ii)} is a direct consequence of the corresponding results for $\alpha=0$, thanks to the point \textit{(i)}. Point \textit{(iv)} follows from the definition, the corresponding point in Proposition \ref{prop:PropertiesinhomSobolevSpaces} and a duality argument provided by the previous point \textit{(ii)}. Point \textit{(v)} follows from points \textit{(iii)} and \textit{(v)} in Proposition \ref{prop:PropertiesinhomSobolevSpaces} and a dilation argument when $\alpha\geqslant 0$, see e.g. \cite[Proposition~2.15]{Gaudin2022}. The case $\alpha<0$ follows by duality thanks to the current point \textit{(ii)}.
\end{proof}

Let us start the construction of corresponding function spaces on the half-line.

\begin{definition} Let $q\in(1,+\infty)$, $\alpha\in\mathbb{R}$, $\mathfrak{h}\in\{\mathrm{H},\dot{\mathrm{H}}\}$. We assume assume moreover that $\alpha<1/q$ when $\mathfrak{h}=\dot{\mathrm{H}}$.
We define by restriction, in the sense of distributions, the normed vector space
\begin{align*}
    \mathfrak{h}^{\alpha,q}(\mathbb{R}_+,X) :=  \mathfrak{h}^{\alpha,q}(\mathbb{R},X)_{|_{\mathbb{R}_+}}\text{.}
\end{align*}
This is a Banach space with respect to the quotient norm
\begin{align*}
    \lVert u \rVert_{\mathfrak{h}^{\alpha,q}(\mathbb{R}_+,X)} := \inf_{\substack{U_{|_{\mathbb{R}_+}}=u\text{,}\\ U\in \mathfrak{h}^{\alpha,q}(\mathbb{R},X)\text{.}}} \lVert U\rVert_{\mathfrak{h}^{\alpha,q}(\mathbb{R},X)}\text{.}
\end{align*}
\end{definition}

\begin{proposition}\label{prop:PropertiesSobolevSpacesR+}Let $q\in(1,+\infty)$, $\alpha\in\mathbb{R}$, $\mathfrak{h}\in\{\mathrm{H},\dot{\mathrm{H}}\}$. We assume assume moreover that $\alpha<1/q$ when $\mathfrak{h}=\dot{\mathrm{H}}$. The following properties hold :
\begin{enumerate}
    \item ${\mathfrak{h}}^{\alpha,q}(\mathbb{R}_+,X)$ is a reflexive Banach space, for which $\eus{S}_0(\mathbb{R},X)_{|_{\mathbb{R}_+}}$ is a dense subspace;
    \item provided $\alpha \in (-1+1/q,1/q)$, for all $u\in {\mathfrak{h}}^{\alpha,q}(\mathbb{R}_+,X)$, the extension of $u$ to the whole line by $0$ denoted by $\Tilde{u}$ yields an element of ${\mathfrak{h}}^{\alpha,q}(\mathbb{R},X)$
    \begin{align*}
        \lVert \Tilde{u}\rVert_{{\mathfrak{h}}^{\alpha,q}(\mathbb{R},X)} \sim_{\alpha,q,X} \lVert u\rVert_{{\mathfrak{h}}^{\alpha,q}(\mathbb{R}_+,X)}\text{;}
    \end{align*}
    \item provided $\alpha \in [0,1/q)$, $\frac{1}{r}=\frac{1}{q}-\alpha$, for all $u\in {\mathfrak{h}}^{\alpha,q}(\mathbb{R}_+,X)$, $v\in {\mathrm{L}}^{r'}(\mathbb{R}_+,X)$,
    \begin{align*}
        &\lVert u\rVert_{\mathrm{L}^{r}(\mathbb{R}_+,X)} \lesssim_{\alpha,q,X} \lVert u\rVert_{{\mathfrak{h}}^{\alpha,q}(\mathbb{R}_+,X)}\text{,}\\
        &\lVert v\rVert_{{\mathfrak{h}}^{-\alpha,q'}(\mathbb{R}_+,X)} \lesssim_{\alpha,q,X} \lVert v\rVert_{\mathrm{L}^{r'}(\mathbb{R}_+,X)}\text{;}
    \end{align*}
    \item for all $\alpha \in (-1+1/q,1/q)$, the subspace $\mathrm{C}_c^\infty(\mathbb{R}_+,X)$ is dense in ${\mathfrak{h}}^{\alpha,q}(\mathbb{R}_+,X)$;
    \item whenever $\alpha\in(-1+1/q,1/q)$,
    \begin{align*}
        ({\mathfrak{h}}^{\alpha,q}(\mathbb{R}_+,X))^\ast = {\mathfrak{h}}^{-\alpha,q'}(\mathbb{R}_+,X^\ast)\text{;}
    \end{align*}
    \item provided $\alpha \in [0,1/q)$,  for all $u\in {\mathrm{H}}^{\alpha,q}(\mathbb{R}_+,X)$, 
    \begin{align*}
        \lVert u\rVert_{\mathrm{L}^{q}(\mathbb{R}_+,X)} + \lVert u\rVert_{\dot{\mathrm{H}}^{\alpha,q}(\mathbb{R}_+,X)} \sim_{\alpha,q,X} \lVert u\rVert_{{\mathrm{H}}^{\alpha,q}(\mathbb{R}_+,X)}\text{;}
    \end{align*}
\end{enumerate}
\end{proposition}

\begin{proof} Point \textit{(i)} follows from the definition of function spaces by restriction, and the properties for their counterparts on $\mathbb{R}$. The point \textit{(ii)} is a direct consequence of point \textit{(v)} from both Propositions~\ref{prop:PropertiesinhomSobolevSpaces}~and~\ref{prop:PropertiesHomSobolevSpaces}.
Point \textit{(iii)} follows from the from the definition of function spaces by restriction and the corresponding result in Propositions~\ref{prop:PropertiesinhomSobolevSpaces}~and~\ref{prop:PropertiesHomSobolevSpaces}.

The point \textit{(iv)} for $\alpha\leqslant0$ follows from point \textit{(i)} and \textit{(iii)}: indeed, both yields that $\mathrm{L}^{r}(\mathbb{R}_+,X)$ is dense in $\mathrm{H}^{\alpha,q}(\mathbb{R}_+,X)$, therefore it suffices to approximate functions in $\mathrm{L}^r(\mathbb{R}_+,X)$ by ones in $\mathrm{C}_c^\infty(\mathbb{R}_+,X)$. For $\alpha> 0$, the inhomogeneous case is known to be true, see for instance \cite[Proposition~6.4]{LindemulderMeyriesVeerar2018}. The case of homogeneous function space follows since by \textit{(i)}, and by construction, ${\mathrm{H}}^{\alpha,q}(\mathbb{R}_+,X)$ embeds continuously and densely in $\dot{\mathrm{H}}^{\alpha,q}(\mathbb{R}_+,X)$.

For the point \textit{(v)}, one may use the \textit{(iv)} and the case $\alpha=0$ in order to reproduce the proof as in the scalar case, e.g. one may reproduce the proof of \cite[Proposition~2.27]{Gaudin2022}.

Finally, the point \textit{(vi)} can be proved by the mean of current point \textit{(ii)} and the point \textit{(iii)} from Proposition \ref{prop:PropertiesinhomSobolevSpaces}.
\end{proof}

The next lemma is nothing but the Hardy-Sobolev inequality in the vector valued setting with homogeneous estimate. Its proof is left to the reader and use a complex interpolation argument allowed by \cite[Theorem~6.7]{LindemulderMeyriesVeerar2018}, then dilation and density arguments by the mean of the points \textit{(iv)} and \textit{(vi)} of Proposition \ref{prop:PropertiesSobolevSpacesR+}.

\begin{lemma}\label{lem:HardySobolevInequality} Let $q\in(1,+\infty)$, $\alpha\in[0,1/q)$. For all $u\in \dot{\mathrm{H}}^{\alpha,q}(\mathbb{R}_+,X)$ the following inequality holds
\begin{align*}
    \left\lVert \tau \mapsto \frac{u(\tau)}{\tau^\alpha}\right\rVert_{\mathrm{L}^q(\mathbb{R}_+,X)} \lesssim_{\alpha,q} \lVert u\rVert_{\dot{\mathrm{H}}^{\alpha,q}(\mathbb{R}_+,X)}\text{.}
\end{align*}
\end{lemma}

\subsection{The derivative on the half-line.}

We will not discuss here the construction and properties of inhomogeneous Bessel potential spaces $\mathrm{H}^{s,p}(\mathbb{R}_+,X)$  for $s\in\mathbb{R}$, $p\in(1,+\infty)$ and the meaning of traces at $0$. Therefore as in the previous subsection, we refer 
 to \cite{MeyriesVeraar2012,LindemulderMeyriesVeerar2018,ScharfSchmeißerSickel2011} for more details.

We recall that one may define the unbounded operator $\frac{\mathrm{d}}{\mathrm{d}t}$ on $\mathrm{L}^q(\mathbb{R}_+,X)$, also denoted by  $\partial_t$, with domain,
\begin{align*}
    \mathrm{D}_{q}(\partial_t):=\mathrm{H}^{1,q}_0(\mathbb{R}_+,X):= \{ f\in \mathrm{L}^q(\mathbb{R}_+,X)\,|\, \partial_t f \in \mathrm{L}^q(\mathbb{R}_+,X)\text{, } f(0)=0\}\text{. }
\end{align*}
Then, thanks to \cite[Theorem~3.1]{DoreVenni1987} (see also \cite[Sections~8.4,~8.5]{bookHaase2006}, \cite[Theorem~6.8]{LindemulderMeyriesVeerar2018} or \cite[Theorem~4.3.14]{PrussSimonett2016}), $\frac{\mathrm{d}}{\mathrm{d}t}$ is an injective sectorial operator on $\mathrm{L}^q(\mathbb{R}_+,X)$ which admits bounded imaginary powers, satisfying, for all $s\in\mathbb{R}$
\begin{align}
    \left\lVert (\partial_t)^{is}\right\rVert_{\mathrm{L}^q(\mathbb{R}_+,X)\rightarrow\mathrm{L}^q(\mathbb{R}_+,X)} \lesssim_{q,X} (1+s^2)e^{\frac{\pi}{2}|s|}\text{,}
\end{align}
implying that $(\partial_t)^{\alpha}$, $\alpha\in[0,1]$, is injective with domain $\mathrm{H}^{\alpha,q}_0(\mathbb{R}_+,X)$, $\alpha\neq \frac{1}{q}$, see \cite[Theorems~6.7~\&~6.8]{LindemulderMeyriesVeerar2018}, and we have an isomorphism, provided $\alpha \in(-1+1/q,1/q)$,
\begin{align}\label{eq:isomDirTimeDeriv}
    (\partial_t)^{\alpha}\,:\,\dot{\mathrm{H}}^{\alpha,q}(\mathbb{R}_+,X)\longrightarrow {\mathrm{L}}^{q}(\mathbb{R}_+,X)\text{. }
\end{align}

For $\beta \in (0,1)$, $\alpha\in[\beta,1]$, $\alpha\neq 1/q$, $\gamma\in[\beta,\alpha]$, the following representation formula holds for all $f\in (\partial_t)^\beta\mathrm{H}^{\alpha,q}_0(\mathbb{R}_+,X)$,
\begin{align}\label{eq:RepFormulaDirichletFracTimeDeriv}
    (\partial_t)^{-\beta}f(t) = \frac{1}{\Gamma(\gamma)} \int_{0}^{t} \frac{1}{(t-\tau)^{1-\gamma}} (\partial_t)^{\gamma-\beta}f(\tau) \,\mathrm{d}\tau, t>0.
\end{align}
Above formula remains true for $f\in \dot{\mathrm{H}}^{\alpha-\beta,q}(\mathbb{R}_+,X)$, provided $\alpha,\alpha-\beta< 1/q$.

Similarly, the "dual" operator $-\frac{\mathrm{d}}{\mathrm{d}t}$, with domain $\mathrm{D}_{q}(-\partial_t):=\mathrm{H}^{1,q}(\mathbb{R}_+,X)$, is an injective sectorial operator on $\mathrm{L}^q(\mathbb{R}_+,X)$ which admits bounded imaginary powers. For all $s\in\mathbb{R}$
\begin{align}
    \left\lVert (-\partial_t)^{is}\right\rVert_{\mathrm{L}^q(\mathbb{R}_+,X)\rightarrow\mathrm{L}^q(\mathbb{R}_+,X)} \lesssim_{q,X} (1+s^2)e^{\frac{\pi}{2}|s|}\text{,}
\end{align}
which also implies injectivity of $(-\partial_t)^\alpha$, $\alpha\in[0,1]$, with domain ${\mathrm{H}}^{\alpha,q}(\mathbb{R}_+,X)$.

For $\alpha \in(-1+1/q,1/q)$, we still have an isomorphism 
\begin{align}\label{eq:isomNeuTimeDeriv}
    (-\partial_t)^{\alpha}\,:\,\dot{\mathrm{H}}^{\alpha,q}(\mathbb{R}_+,X)\longrightarrow {\mathrm{L}}^{q}(\mathbb{R}_+,X)\text{. }
\end{align}

For $\beta \in (0,1)$, $\alpha\in[\beta,1]$, $\gamma\in[\beta,\alpha]$, the following representation formula holds for all $f\in (-\partial_t)^{\beta}\mathrm{H}^{\alpha,q}(\mathbb{R}_+,X)$,
\begin{align}\label{eq:RepFormulaNeumFracMinusTimeDeriv}
    (-\partial_t)^{-\beta}f(t) = \frac{1}{\Gamma(\gamma)} \int_{t}^{+\infty} \frac{1}{(\tau-t)^{1-\gamma}} (-\partial_t)^{\gamma-\beta}f(\tau) \,\mathrm{d}\tau, t>0.
\end{align}
Above formula remains true for $f\in \dot{\mathrm{H}}^{\alpha-\beta,q}(\mathbb{R}_+,X)$, provided $\alpha,\alpha-\beta< 1/q$.

More details about the functional analytic properties of operators $\partial_t$ and $-\partial_t$ can also be found in \cite[Section~3.2]{PrussSimonett2016} and \cite[Section~6]{LindemulderMeyriesVeerar2018}, where the case of (power-)weighted, but inhomogeneous, Sobolev spaces have been widely treated.

\subsection{A comment for homogeneous Sobolev spaces and the time derivative on a finite interval}

We finish this section with a discussion about Sobolev space on $(0,T)$ (or $[0,T]$), $T>0$, and the related derivative operators. On can define those space similarly.

\begin{definition} Let $q\in(1,+\infty)$, $-1+1/q<\alpha<1/q$ when $\mathfrak{h}=\dot{\mathrm{H}}$.
We define by restriction, in the sense of distributions, the normed vector space
\begin{align*}
    \mathfrak{h}^{\alpha,q}((0,T),X) :=  \mathfrak{h}^{\alpha,q}(\mathbb{R},X)_{|_{(0,T)}}\text{.}
\end{align*}
with the induced quotient norm.
\end{definition}

But since $\mathbbm{1}_{(0,T)}= \mathbbm{1}_{(0,+\infty)}-\mathbbm{1}_{\mathbb{R}_+}(\cdot-T)$, we obtain

\begin{proposition}\label{prop:PropertiesSobolevSpaces0T}Let $q\in(1,+\infty)$, $\alpha\in(-1+1/q,1/q)$. The following properties hold :
\begin{enumerate}
    \item $\dot{\mathrm{H}}^{\alpha,q}((0,T),X)$ is a reflexive Banach space, for which $\eus{S}_0(\mathbb{R},X)_{|_{(0,T)}}$ is a dense subspace;
    \item for all $u\in \dot{\mathrm{H}}^{\alpha,q}((0,T),X)$, the extension of $u$ to the whole line denoted $\Tilde{u}$, is such that
    \begin{align*}
        \lVert \Tilde{u}\rVert_{\dot{\mathrm{H}}^{\alpha,q}(\mathbb{R}_+,X)}+\lVert \Tilde{u}\rVert_{{\mathrm{H}}^{\alpha,q}(\mathbb{R},X)} \lesssim_{\alpha,q,X,T} \lVert u\rVert_{\dot{\mathrm{H}}^{\alpha,q}((0,T),X)}\text{.}
    \end{align*}
    \item $\dot{\mathrm{H}}^{\alpha,q}((0,T),X)={\mathrm{H}}^{\alpha,q}((0,T),X)$ with equivalence of norms (depending on $T$).
\end{enumerate}
\end{proposition}

From there, and in particular from point \textit{(ii)} of Proposition \ref{prop:PropertiesSobolevSpaces0T}, one may expect that the theory on the half line will carry over the behavior on $(0,T)$ up to extend the elements by $0$, or up to the multiplication by $\mathbbm{1}_{(0,T)}$. And this is indeed, what actually happens for $\partial_t$ and $-\partial_t$ according to \cite[Section~8.5]{bookHaase2006}.


\section{The global-in-time \texorpdfstring{$\dot{\mathrm{H}}^{\alpha,q}$}{Haq}-maximal regularity with homogeneous trace estimate}

Now, we go back to $\mathrm{L}^q$-maximal regularity on a UMD Banach space $X$. We are going to state few minor improvements of above results, the first one is about global-in-time estimates when the initial data $u_0$ lies in the homogeneous space $\mathring{\eus{D}}_{A}(\theta,q)$, provided $\theta\in(0,1)$, $q\in(1,+\infty)$. 

A second goal is to obtain a $\dot{\mathrm{H}}^{\alpha,q}$-maximal regularity result as a variation of above Theorem~\ref{thm:LqMaxRegUMD} where we take the advantage of \cite[Proposition~2.4]{Pruss2002}, and the isomorphism properties \eqref{eq:isomDirTimeDeriv} and \eqref{eq:isomNeuTimeDeriv}. Our proof for the corresponding homogeneous trace estimates is mainly inspired by techniques from the proofs of \cite[Lemma~2.19,~Theorem~2.20]{DanchinHieberMuchaTolk2020}, see also \cite[Section~3.4]{PrussSimonett2016} for similar estimates proven in a similar way.

\subsection{About mild solutions in the context of homogeneous operator theory}

In the literature, it seems difficult to have a clear and definitive mention of what would be the exact meaning of a mild solution of \eqref{ACP} in the context of homogeneous functions spaces with respect the space variable (here the roles are played by $\mathrm{D}(\mathring{A})$ and $\mathring{\eus{D}}_A(\theta,q)$). Here is an attempt.

\begin{definition} Let $\omega\in [0,\frac{\pi}{2})$, $(\mathrm{D}(A),A)$ an $\omega$-sectorial operator on a UMD Banach space $X$, such that it satisfies Assumptions \eqref{asmpt:homogeneousdomaindef} and \eqref{asmpt:homogeneousdomainintersect}. 

Let $T\in(0,+\infty]$, $f\in\mathrm{L}^1_{loc}([0,T),X)$ and $u_0\in X+ \mathrm{D}(\mathring{A})$. We say that $u\,:\,[0,T)\longrightarrow X+ \mathrm{D}(\mathring{A})$ is a \textbf{homogeneous}-mild solution of \eqref{ACP} if
\begin{enumerate}[label=\textit{(\roman*)}]
    \item $u\in\mathrm{C}^0_b([0,T),X+\mathrm{D}(\mathring{A}))$,
    \item $v(t):=u(t)-e^{-tA}u_0 \in X$, for all $t\in[0,T)$,
    \item $v\in \mathrm{C}^0_b([0,T),X)$ is a mild solution of (ACP${}_0$) in the classical sense, \textit{i.e.} for all $t\in[0,T)$,
    \begin{align*}
        \int_{0}^{t} v(s)\,\mathrm{d}s\in \mathrm{D}(A)\text{,}
    \end{align*}
    and
    \begin{align*}
        v(t) + A\int_{0}^{t} v(s)\,\mathrm{d}s = \int_{0}^{t} f(s)\,\mathrm{d}s \text{ in } X \text{.}
    \end{align*}
\end{enumerate}
\end{definition}

\begin{proposition}Let $\omega\in [0,\frac{\pi}{2})$, $(\mathrm{D}(A),A)$ an $\omega$-sectorial operator on a UMD Banach space $X$, such that it satisfies Assumptions \eqref{asmpt:homogeneousdomaindef} and \eqref{asmpt:homogeneousdomainintersect}. Let $T\in(0,+\infty]$, $f\in\mathrm{L}^1_\mathrm{loc}([0,T),X)$ and $u_0\in X+ \mathrm{D}(\mathring{A})$.

The problem \eqref{ACP} admits at most one \textbf{homogeneous}-mild solution.
\end{proposition}

\begin{proof}Let $u_1$ and $u_2$ be two \textbf{homogeneous}-mild solutions to \eqref{ACP}. Then, we set for all $t\geqslant0$, $V(t):=u_1(t)-u_2(t) = (u_1(t) -e^{-tA}u_0) - (u_2(t) -e^{-tA}u_0)$. It follows that $V$ is a mild solution of (ACP${}_0^0$) in the classical sense. Hence, uniqueness provided by \cite[Proposition~3.1.16]{ArendtBattyHieberNeubranker2011} yields $V=0$ in $X$.
\end{proof}

\begin{remark}Since for $u_0\in X$, one recovers classical mild solutions from the definition, and there is no ambiguity since mild solutions in the classical sense are in particular \textbf{homogeneous}-mild solution. From now on, we will refer without distinction to \textbf{homogeneous}-mild solution and mild solution in the classical sense, as mild solutions.
\end{remark}

\subsection{Preliminary lemmas}

First, we state a Lemma for the problem (ACP${}^0$), about homogeneous fractional Sobolev in-time estimates for initial data $u_0\in \mathring{\eus{D}}_{A}(\theta,q)$, $q\in(1,+\infty)$, $\theta\in(0,1)$.

\begin{lemma}\label{lem:homdataSobest} Let $\omega\in [0,\frac{\pi}{2})$, $(\mathrm{D}(A),A)$ an $\omega$-sectorial operator on a UMD Banach space $X$, such that it satisfies Assumptions \eqref{asmpt:homogeneousdomaindef} and \eqref{asmpt:homogeneousdomainintersect}. Let $q\in(1,+\infty)$, $\alpha\in(-1+1/q,1/q)$.

For all $u_0 \in \mathring{\eus{D}}_{A}(1+\alpha-{1}/{q},q)$, we have
\begin{align*}
    t\mapsto \mathring{A}e^{-tA}u_0 \in \dot{\mathrm{H}}^{\alpha,q}(\mathbb{R}_+,X)\cap\mathrm{L}^q_{1+\alpha}(\mathbb{R}_+,X)
\end{align*}
with estimates
\begin{align*}
    \lVert t\mapsto t^{1-(1+\alpha)}\mathring{A}e^{-t A}u_0\rVert_{{\mathrm{L}}^{q}(\mathbb{R}_+,X)}=\lVert u_0\rVert_{\mathring{\eus{D}}_{A}(1+\alpha-{1}/{q},q)}\text{, }\\
    \lVert \mathring{A}e^{-(\cdot) A}u_0\rVert_{\dot{\mathrm{H}}^{\alpha,q}(\mathbb{R}_+,X)} \lesssim_{q,\alpha,A} \lVert u_0\rVert_{\mathring{\eus{D}}_{A}(1+\alpha-{1}/{q},q)}\text{.} 
\end{align*}
\end{lemma}

\begin{proof} We just have to prove the estimate in Sobolev space. The equality of norms is straightforward by definition of the $\mathring{\eus{D}}_{A}(1+\alpha-{1}/{q},q)$-norm.

\textbf{Step 1:} The case $\alpha=0$ is straightforward.

\textbf{Step 2:} The case $\alpha\in(0,1/q)$. For $u_0\in \mathring{\eus{D}}_{A}(1+\alpha-{1}/{q},q)$, one can write $u_0=x_0+a_0$, where $(x_0,a_0)\in X\times \mathrm{D}(\mathring{A})$. By \cite[Proposition~2.6]{DanchinHieberMuchaTolk2020}, the following equality holds in $X$, for all $t > 0$,
\begin{align*}
   \mathring{A} e^{-tA}u_0= Ae^{-tA}x_0 + e^{-tA}\mathring{A}a_0 \text{.}
\end{align*}
Therefore, thanks to the representation formulae \eqref{eq:RepFormulaNeumFracMinusTimeDeriv}, and integral formulations for fractional powers of $A$, we have for all $t>0$,
\begin{align*}
     A^\alpha \mathring{A} e^{-tA}u_0   &=  A^\alpha (Ae^{-tA}x_0 + e^{-tA}\mathring{A}a_0)\\
                                        &= A^{1+\alpha}e^{-tA}x_0 + A^\alpha e^{-tA}\mathring{A}a_0\\
                                        &= \frac{1}{\Gamma(1-\alpha)} \int_{t}^{+\infty} \frac{1}{(\tau-t)^{\alpha}} (A^2e^{-\tau A}x_0 + Ae^{-\tau A}\mathring{A}a_0)\, \mathrm{d}\tau \\
                                        &= (-\partial_t)^{\alpha -1} [ A^2e^{-(\cdot) A}x_0 + Ae^{-(\cdot)A}\mathring{A}a_0 ](t)\\
                                        &= (-\partial_t)^{\alpha -1} [ A \mathring{A} e^{-(\cdot) A}u_0 ](t)\\
                                        &= (-\partial_t)^{\alpha} [\mathring{A} e^{-(\cdot) A}u_0 ](t) \text{.}
\end{align*}
So that, by the isomorphism property \eqref{eq:isomNeuTimeDeriv}, we have
\begin{align*}
    \lVert \mathring{A}e^{-(\cdot) A}u_0\rVert_{\dot{\mathrm{H}}^{\alpha,q}(\mathbb{R}_+,X)}\sim_{\alpha,q} \lVert A^{\alpha}\mathring{A}e^{-(\cdot) A}u_0\rVert_{{\mathrm{L}}^{q}(\mathbb{R}_+,X)}\text{ . }
\end{align*}
From there, we obtain
\begin{align*}
    \lVert A^{\alpha}\mathring{A}e^{-(\cdot) A}u_0\rVert_{{\mathrm{L}}^{q}(\mathbb{R}_+,X)} &= \left(\int_0^{+\infty} (\tau^{\frac{1}{q}}\lVert A^{\alpha}\mathring{A}e^{-\tau A}u_0\rVert_X )^q \frac{\mathrm{d}\tau}{\tau} \right)^{\frac{1}{q}}\\
    &\lesssim_{q,\alpha,A} \left(\int_0^{+\infty} (\tau^{\frac{1}{q}-\alpha}\lVert \mathring{A}e^{-\frac{\tau}{2} A}u_0\rVert_X )^q \frac{\mathrm{d}\tau}{\tau} \right)^{\frac{1}{q}}\\
    &\lesssim_{q,\alpha,A} \lVert u_0\rVert_{\mathring{\eus{D}}_{A}(1+\alpha-{1}/{q},q)} \text{.} 
\end{align*}
Our last set of inequalities follows from the analyticity of the semigroup $(e^{-tA})_{t>0}$ on $X$ and the fact that one can write for all $\tau>0$,
\begin{align}\label{eq:powerssemgrpdecomptover2}
    A^{\alpha}\mathring{A}e^{-\tau A}u_0 = A^{\alpha}e^{-\frac{\tau}{2} A}\mathring{A}e^{-\frac{\tau}{2} A}u_0 \text{ . }
\end{align}
\textbf{Step 3:} The case $\alpha\in(-1+1/q,0)$. We play with the integral representations like \eqref{eq:RepFormulaNeumFracMinusTimeDeriv} and fractional powers of $A$, so that as in \textbf{Step 2}, we should be able to write for $t>0$,
\begin{align*}
     (-\partial_t)^{\alpha}[\mathring{A}e^{-(\cdot) A}u_0](t) &= \int_{t}^{+\infty} (-\partial_\tau)^{\alpha+1}[\mathring{A}e^{-(\cdot) A}u_0](\tau) \,\mathrm{d}\tau \label{eq:}\\ &= \int_{t}^{+\infty} \tau A^{1+\alpha}\mathring{A}e^{-\tau A}u_0\frac{\mathrm{d}\tau}{\tau} \text{ . }
\end{align*}
Notice that the last integral can be understood as an improper Riemann integral, so that it gives a measurable function with values in $X$.

Therefore, we can bound, thanks to the Fatou Lemma and then to the analyticity of the semigroup (we use the same trick \eqref{eq:powerssemgrpdecomptover2}),
\begin{align*}
    \left\lVert t\mapsto\int_{t}^{+\infty} \tau A^{1+\alpha}\mathring{A}e^{-\tau A}u_0\frac{\mathrm{d}\tau}{\tau}\right\rVert_{\mathrm{L}^q(\mathbb{R}_+,X)}^q &\leqslant \liminf_{M\rightarrow +\infty}\int_{0}^{+\infty}\left\lVert \int_{t}^{M} \tau A^{1+\alpha}\mathring{A}e^{-\tau A}u_0\frac{\mathrm{d}\tau}{\tau} \right\rVert_X^q\mathrm{d}{t}\\
    &\leqslant \int_{0}^{+\infty}\left( \int_{t}^{+\infty} \lVert\tau A^{1+\alpha}\mathring{A}e^{-\tau A}u_0\rVert_X\frac{\mathrm{d}\tau}{\tau} \right)^q\mathrm{d}{t}\\
    &\lesssim_{q,\alpha,A} \int_{0}^{+\infty}\left(t^{\frac{1}{q}} \int_{t}^{+\infty} \tau^{-\alpha} \lVert\mathring{A}e^{-\frac{\tau}{2} A}u_0\rVert_X\frac{\mathrm{d}\tau}{\tau} \right)^q\frac{\mathrm{d}{t}}{t} \text{ . }
\end{align*}
Finally, by the mean of Hardy's inequality \cite[Lemma~6.2.6]{bookHaase2006}, we conclude
\begin{align*}
    \left\lVert t\mapsto\int_{t}^{+\infty} \tau A^{1+\alpha}\mathring{A}e^{-\tau A}u_0\frac{\mathrm{d}\tau}{\tau}\right\rVert_{\mathrm{L}^q(\mathbb{R}_+,X)}^q &\lesssim_{q,\alpha,A} \int_{0}^{+\infty}\left(t^{\frac{1}{q}-\alpha} \lVert\mathring{A}e^{-\frac{t}{2} A}u_0\rVert_X \right)^q\frac{\mathrm{d}{t}}{t} \\
    &\lesssim_{q,\alpha,A} \lVert u_0\rVert_{\mathring{\eus{D}}_{A}(1+\alpha-{1}/{q},q)}^q \text{.} 
\end{align*}
\end{proof}

Now, the next lemma ensures that the maximal regularity operator applied to a Sobolev in-time function, even with negative regularity, still yields an actual measurable function with values in $X$.

\begin{lemma}\label{lem:FromHaqToLqloc} Let $\omega\in [0,\frac{\pi}{2})$, $(\mathrm{D}(A),A)$ an $\omega$-sectorial operator on a UMD Banach space $X$, and let $q\in(1,+\infty)$, $\alpha\in[0,1/q)$. 

For $f\in\dot{\mathrm{H}}^{\alpha,q}(\mathbb{R}_+,X)$, the following holds for all $T>0$,
\begin{align*}
    t\mapsto \int_{0}^t e^{-(t-s)A}f(s) \,\mathrm{d}s \in \mathrm{C}^{0}([0,T],X)
\end{align*}
with estimate
\begin{align*}
    \Bigg\lVert t\mapsto \int_{0}^t e^{-(t-s)A}f(s) \,\mathrm{d}s\Bigg\rVert_{\mathrm{L}^\infty([0,T],X)} &\lesssim_{A,\alpha,q} T^{1+\alpha -1/q} \lVert f \rVert_{\dot{\mathrm{H}}^{\alpha,q}(\mathbb{R}_+,X)}\text{.}
\end{align*}
Moreover, if $A$ has the $\mathrm{L}^q$-maximal regularity property, the results still holds for $\alpha\in(-1+1/q,0)$.
\end{lemma}

\begin{proof}\textbf{Step 1:} For $\alpha=0$, $q\in(1,+\infty)$. Let $f\in\mathrm{L}^q(\mathbb{R}_+,X)$. By uniform boundedness of the semigroup $(e^{-tA})_{t\geqslant0}$ on $X$ and H\"{o}lder's inequality yield
\begin{align*}
    \Bigg\lVert \int_{0}^{t}{e^{-(t-s)A}}f(s)\,\mathrm{d}s\Bigg\rVert_{X} &\lesssim_{A}  \int_{0}^{t} \lVert f(s) \rVert_{X}\,\mathrm{d}s\\
    &\lesssim_{A} t^{1-1/q} \lVert f \rVert_{\mathrm{L}^q([0,t],\mathbb{R})}
\end{align*}
The supremum on $t\in[0,T]$ yields the estimate
\begin{align*}
    \Bigg\lVert t\mapsto \int_{0}^t e^{-(t-s)A}f(s) \,\mathrm{d}s\Bigg\rVert_{\mathrm{L}^\infty([0,T],X)} &\lesssim_{A} T^{1 -1/q} \lVert f \rVert_{{\mathrm{L}}^{q}(\mathbb{R}_+,X)}\text{.}
\end{align*}

Continuity in-time follows from the dominated convergence theorem.

\textbf{Step 2:} For $\alpha\in(0,1/q)$, $f\in\dot{\mathrm{H}}^{\alpha,q}(\mathbb{R}_+,X)$, for $\frac{1}{r}=\frac{1}{q}-\alpha$, we have $f\in \mathrm{L}^r(\mathbb{R}_+,X)$ by Sobolev embeddings. Therefore by \textbf{Step 1},
\begin{align*}
    \Bigg\lVert t\mapsto \int_{0}^t e^{-(t-s)A}f(s) \,\mathrm{d}s\Bigg\rVert_{\mathrm{L}^\infty([0,T],X)} &\lesssim_{A} T^{1 -1/r} \lVert f \rVert_{{\mathrm{L}}^{r}(\mathbb{R}_+,X)}\\
     &\lesssim_{A,\alpha,q} T^{1+\alpha -1/q} \lVert f \rVert_{\dot{\mathrm{H}}^{\alpha,q}(\mathbb{R}_+,X)}\text{.}
\end{align*}

\textbf{Step 3:} Let $\alpha\in(-1+1/q,0)$, $f\in\dot{\mathrm{H}}^{\alpha,q}(\mathbb{R}_+,X)$ and assume that $A$ has the maximal regularity property. First of all, by commutation properties for resolvents of $\partial_t$ and $A$, one can write
\begin{align*}
    (\partial_t+A)^{-1}f &= (\partial_t)^{-\alpha - 1} \partial_t (\partial_t+A)^{-1}(\partial_t)^{\alpha} f\\
    &= (\partial_t)^{-\alpha - 1} (\partial_t)^{\alpha} f - (\partial_t)^{-\alpha - 1}A (\partial_t+A)^{-1} (\partial_t)^{\alpha} f\text{.}
\end{align*}
So that setting $f^\alpha := (\partial_t)^{\alpha} f$ and $u^\alpha := (\partial_t+A)^{-1}f^\alpha$, by the representation formula \eqref{eq:RepFormulaDirichletFracTimeDeriv}, we end up with the following expression for $t>0$
\begin{align*}
    (\partial_t+A)^{-1}f(t) =\int_{0}^t e^{-(t-s)A}f(s) \,\mathrm{d}s = \frac{1}{\Gamma(1+\alpha)}\int_{0}^{t} \frac{1}{(t-s)^{-\alpha}}[ f^\alpha(s) - Au^\alpha(s)]\, \mathrm{d}s\text{.}
\end{align*}
Young's inequality for the convolution, then the triangle inequality yield
\begin{align*}
    \Bigg\lVert t\mapsto \int_{0}^t e^{-(t-s)A}f(s) \,\mathrm{d}s\Bigg\rVert_{\mathrm{L}^\infty([0,T],X)} &\leqslant \frac{1}{\Gamma(1+\alpha)} \left\lVert t\mapsto t^{\alpha}  \right\rVert_{\mathrm{L}^{q'}([0,T])}\lVert f^{\alpha} -Au^\alpha \rVert_{\mathrm{L}^q(\mathbb{R}_+,X)}\\
    &\leqslant \frac{T^{1+\alpha-1/q}}{\Gamma(1+\alpha)(\alpha q' +1)^\frac{1}{q'}} \left(\lVert f^\alpha \rVert_{{\mathrm{L}}^{q}(\mathbb{R}_+,X)} + \lVert Au^\alpha \rVert_{\mathrm{L}^q(\mathbb{R}_+,X)}\right) \text{.}
\end{align*}
From there, we recall that we have assumed the $\mathrm{L}^q$-maximal regularity property, so that, by the isomorphism property \eqref{eq:isomDirTimeDeriv},
\begin{align*}
    \Bigg\lVert t\mapsto \int_{0}^t e^{-(t-s)A}f(s) \,\mathrm{d}s\Bigg\rVert_{\mathrm{L}^\infty([0,T],X)} &\lesssim_{A,\alpha,q} {T^{1+\alpha-1/q}} \lVert f^\alpha \rVert_{{\mathrm{L}}^{q}(\mathbb{R}_+,X)} \\
    &\lesssim_{A,\alpha,q} {T^{1+\alpha-1/q}} \lVert f \rVert_{\dot{\mathrm{H}}^{\alpha,q}(\mathbb{R}_+,X)} \text{.}
\end{align*}
\end{proof}

\begin{corollary}\label{cor:ContinuityDAthetaq}Let $\omega\in [0,\frac{\pi}{2})$, $(\mathrm{D}(A),A)$ an $\omega$-sectorial operator on a UMD Banach space $X$, and let $q\in(1,+\infty)$, $\alpha\in[0,1/q)$. 

For $f\in\dot{\mathrm{H}}^{\alpha,q}(\mathbb{R}_+,X)$, the following holds for all $T>0$,
\begin{align*}
    t\mapsto \int_{0}^t e^{-(t-s)A}f(s) \,\mathrm{d}s \in \mathrm{C}^{0}([0,T],\eus{D}_{A}(1+\alpha-1/q,q))
\end{align*}
with estimate, for all $T>0$,
\begin{align}
    \Bigg\lVert t\mapsto \int_{0}^t e^{-(t-s)A}f(s) \,\mathrm{d}s\Bigg\rVert_{\mathrm{L}^\infty([0,T],{\eus{D}}_{A}(1+\alpha-1/q,q))} &\lesssim_{A,\alpha,q} (1+ {T^{1+\alpha-1/q}}) \lVert f \rVert_{\dot{\mathrm{H}}^{\alpha,q}(\mathbb{R}_+,X)}\label{eq:estContinuityDAthetaqT}\\
    \Bigg\lVert t\mapsto \int_{0}^t e^{-(t-s)A}f(s) \,\mathrm{d}s\Bigg\rVert_{\mathrm{L}^\infty(\mathbb{R}_+,\mathring{\eus{D}}_{A}(1+\alpha-1/q,q))} &\lesssim_{A,\alpha,q} \lVert f \rVert_{\dot{\mathrm{H}}^{\alpha,q}(\mathbb{R}_+,X)}\text{.}\label{eq:estContinuityDAthetaqinfty}
\end{align}
\end{corollary}

\begin{proof} Thanks to Lemma \ref{lem:FromHaqToLqloc}, it suffices to prove the estimate \eqref{eq:estContinuityDAthetaqinfty}.

\textbf{Step 1:} First we assume $\alpha=0$ and $f\in\mathrm{L}^q(\mathbb{R}_+,X)$, we may extend $f$ to $\mathbb{R}$ by setting $f(t):=0$ for $t<0$. In a similar fashion to what has been done in \cite[Lemma~2.19]{DanchinHieberMuchaTolk2020}, we can bound
\begin{align*}
    \Bigg\lVert \int_{0}^{t}e^{-sA}f(t-s)\,\mathrm{d}s\Bigg\rVert_{\mathring{\eus{D}}_{A}(1-{1}/{q},q)}^q &= \int_{0}^{+\infty}\Bigg( \tau^\frac{1}{q} \Big\lVert A \int_{0}^{t} e^{-(\tau+s)A}f(t-s)\,\mathrm{d}s\Big\rVert_X\Bigg)^q \frac{\mathrm{d}\tau}{\tau}\\
    &\lesssim_{A,q} \int_{0}^{+\infty}\Bigg( \int_{0}^{+\infty} \frac{1}{\tau+s}\lVert f(t-s)\rVert_X\mathrm{d}s\Bigg)^q {\mathrm{d}\tau}\\
    &\lesssim_{A,q} \int_{0}^{+\infty}\Bigg( \tau^{\frac{1}{q}-1} \int_{0}^{\tau} \lVert f(t-s)\rVert_X\mathrm{d}s\Bigg)^q \frac{\mathrm{d}\tau}{\tau}\\
    &\quad + \int_{0}^{+\infty} \Bigg( \tau^{\frac{1}{q}}\int_{\tau}^{+\infty} \frac{1}{s}\lVert f(t-s)\rVert_X\mathrm{d}s \Bigg)^{q} \frac{\mathrm{d}\tau}{\tau} \text{.}
\end{align*}
We can apply Hardy's inequalities, see \cite[Lemma~6.2.6]{bookHaase2006}, to obtain
\begin{align*}
    \Bigg\lVert \int_{0}^{t}e^{-sA}f(t-s)\,\mathrm{d}s\Bigg\rVert_{\mathring{\eus{D}}_{A}(1-{1}/{q},q)}^q \lesssim_{A,q} \int_{0}^{+\infty}( \tau^{\frac{1}{q}} \lVert f(t-\tau)\rVert_X)^q \frac{\mathrm{d}\tau}{\tau}
    \lesssim_{A,q} \lVert f\rVert_{\mathrm{L}^q(\mathbb{R}_+,X)}^q\text{. }
\end{align*}

\textbf{Step 2:} For $\alpha\in(0,1/q)$, $f\in \dot{\mathrm{H}}^{\alpha,q}(\mathbb{R}_+,X)$, by Sobolev embeddings, we have $f\in \mathrm{L}^r(\mathbb{R}_+,X)$, $r= \frac{q}{1-\alpha q}$, so that by \textbf{Step 1}, for all $T>0$ :
\begin{align*}
    t\mapsto \int_{0}^{t}e^{-sA}f(t-s)\,\mathrm{d}s \in \mathrm{C}^0([0,T],{\eus{D}}_{A}(1-{1}/{r},r))\text{ . }
\end{align*}
So that it is well defined. Let $t>0$, by \cite[Lemma~2.15]{DanchinHieberMuchaTolk2020}, we have
\begin{align*}
    \Bigg\lVert \int_{0}^{t}e^{-sA}f(t-s)\,\mathrm{d}s\Bigg\rVert_{\mathring{\eus{D}}_{A}(1+\alpha-{1}/{q},q)}^q \sim_{\alpha,q} \underbrace{\int_{0}^{+\infty}\Bigg( \tau^{1+\frac{1}{q}-\alpha} \Big\lVert A^2e^{-\tau A} \int_{0}^{t} e^{-(t-s)A}f(s)\,\mathrm{d}s\Big\rVert_X\Bigg)^q \frac{\mathrm{d}\tau}{\tau}}_{(I)} \text{ . }
\end{align*}
Since $\partial_t$ and $A$ have commuting resolvents, we have 
\begin{align*}
    (\partial_t+A)^{-1} = (\partial_t)^{-\alpha}(\partial_t+A)^{-1}(\partial_t)^{\alpha}\text{.}
\end{align*}
Therefore, setting $f^{\alpha}:=(\partial_t)^{\alpha}f \in\mathrm{L}^q(\mathbb{R},X)$ (up to consider, again, the extension of $f^{\alpha}$ (not $f$) to the whole line by $0$), we can use the representation formula \eqref{eq:RepFormulaDirichletFracTimeDeriv}, to obtain
\begin{align*}
    (I) \sim_{\alpha,q} \int_{0}^{+\infty}\Bigg( \tau^{1+\frac{1}{q}-\alpha} \Big\lVert \int_{0}^{t} \frac{1}{(t-u)^{1-\alpha}}\int_{0}^{u} A^2e^{-(\tau+(u-s))A}f^{\alpha}(s)\,\mathrm{d}s\,\mathrm{d}u\Big\rVert_X\Bigg)^q \frac{\mathrm{d}\tau}{\tau} \text{ . }
\end{align*}
From there, we can use the triangle inequality and we can write, provided $ 0 \leqslant s \leqslant u \leqslant t$,
\begin{align*}
    A^2e^{-(\tau+(u-s))A}=A^{1+\alpha}e^{-\frac{(\tau+(u-s))}{2}A}A^{1-\alpha}e^{-\frac{\tau}{2}A} e^{-\frac{(u-s)}{2}A} \text{,}
\end{align*}
so that, by analyticity of the semigroup $(e^{-tA})_{t\geqslant 0}$ on $X$, and the Fubini-Tonelli theorem, we have
\begin{align*}
    (I) \lesssim_{\alpha,q,A} \int_{0}^{+\infty}\Bigg( \tau^{\frac{1}{q}}  \int_{0}^{t}\int_{0}^{u} \frac{1}{(t-u)^{1-\alpha}}\frac{1}{(\tau+(u-s))^{1+\alpha}}\lVert f^{\alpha}(s)\rVert_X\mathrm{d}s\,\mathrm{d}u\Bigg)^q \frac{\mathrm{d}\tau}{\tau} \text{ . }
\end{align*}
Again by Fubini-Tonelli, and since $\int_{s}^{t}\frac{1}{(t-u)^{1-\alpha}}\frac{1}{(\tau+(u-s))^{1+\alpha}} \mathrm{d}u = \frac{1}{\alpha}\frac{(t-s)^{\alpha}}{(\tau+(t-s))\tau^{\alpha}}$, it follows that
\begin{align*}
    (I) \lesssim_{\alpha,q,A} \int_{0}^{+\infty}\Bigg( \tau^{\frac{1}{q}-\alpha}  \int_{0}^{t}\frac{s^{\alpha}}{(\tau+s)}\lVert f^{\alpha}(t-s)\rVert_X\mathrm{d}s\Bigg)^q \frac{\mathrm{d}\tau}{\tau} \text{ . }
\end{align*}
We can reproduce the use of Hardy's inequalities \cite[Lemma~6.2.6]{bookHaase2006} as in \textbf{Step 1}, to obtain
\begin{align*}
    \int_{0}^{+\infty}\Bigg( \tau^{\frac{1}{q}-\alpha}  \int_{0}^{+\infty}\frac{s^{\alpha}}{(\tau+s)}\lVert f^{\alpha}(t-s)\rVert_X\mathrm{d}s\Bigg)^q \frac{\mathrm{d}\tau}{\tau} \lesssim_{\alpha,q} \lVert f^\alpha\rVert_{\mathrm{L}^q(\mathbb{R}_+,X)}^q \lesssim_{\alpha,q}  \lVert f\rVert_{\dot{\mathrm{H}}^{\alpha,q}(\mathbb{R}_+,X)}^q \text{ .}
\end{align*}
One may also prove the continuity in-time by a density argument and the estimate \eqref{eq:estContinuityDAthetaqT}.
\end{proof}

\subsection{The main result}

\begin{theorem}\label{thm:LqMaxRegUMDHomogeneous}Let $\omega\in [0,\frac{\pi}{2})$, $(\mathrm{D}(A),A)$ an $\omega$-sectorial operator on a UMD Banach space $X$, such that it satisfies Assumptions \eqref{asmpt:homogeneousdomaindef} and \eqref{asmpt:homogeneousdomainintersect}. Let $q\in(1,+\infty)$, $\alpha\in(-1+1/q,1/q)$ and assume that one of the two following conditions is satisfied
\begin{enumerate}
    \item $\alpha\geqslant 0$ and $A$ has the $\mathrm{L}^q$-maximal regularity property,
    \item $\alpha < 0$ and $A$ has BIP on $X$ of type $\theta_A<\frac{\pi}{2}$.
\end{enumerate}

 Let $T\in(0,+\infty]$. For $f\in\dot{\mathrm{H}}^{\alpha,q}((0,T),X)$, $u_0\in \mathring{\eus{D}}_{A}(1+\alpha-{1}/{q},q)$, the problem \eqref{ACP} admits a unique mild solution $u\in \mathrm{C}^0_b([0,T],\mathring{\eus{D}}_{A}(1+\alpha-{1}/{q},q))$ such that $\partial_t u$, $Au \in \dot{\mathrm{H}}^{\alpha,q}((0,T),X)$ with estimate
\begin{align}\label{eq:BoundLqMaxReghomogeneous}
    \lVert u \rVert_{\mathrm{L}^\infty([0,T],\mathring{\eus{D}}_{A}(1+\alpha-{1}/{q},q))} \lesssim_{A,q,\alpha} \lVert (\partial_t u, Au)\rVert_{\dot{\mathrm{H}}^{\alpha,q}((0,T),X)} \lesssim_{A,q,\alpha} \lVert f\rVert_{\dot{\mathrm{H}}^{\alpha,q}((0,T),X)} + \lVert u_0\rVert_{\mathring{\eus{D}}_{A}(1+\alpha-{1}/{q},q)}\text{. }
\end{align}
Moreover, if $A$ admits BIP on $X$ of type $\theta_A<\frac{\pi}{2}$, for $f\in \dot{\mathrm{H}}^{\alpha,q}((0,T),X)$, $u_0\in {\eus{D}}_{A}(1+\alpha-{1}/{q},q)$ and all $\beta\in[0,1]$,
\begin{align}\label{eq:mixedDerivEstimate}
    \lVert (-\partial_t)^{1-\beta} A^{\beta} u\rVert_{\dot{\mathrm{H}}^{\alpha,q}((0,T),X)} \lesssim_{A,q,\alpha} \lVert f\rVert_{\dot{\mathrm{H}}^{\alpha,q}((0,T),X)} + \lVert u_0\rVert_{\mathring{\eus{D}}_{A}(1+\alpha-{1}/{q},q)} \text{. }
\end{align}
\end{theorem}

\begin{remark} $\bullet$ In Theorem \ref{thm:LqMaxRegUMDHomogeneous}, assumptions \eqref{asmpt:homogeneousdomaindef} and \eqref{asmpt:homogeneousdomainintersect} are assumed here in order to ensure that $\mathring{\eus{D}}_{A}(\theta,q)$ is a well defined, even if not complete, normed vector space.

$\bullet$ If $u_0=0$, the estimate \eqref{eq:mixedDerivEstimate} remains valid if we replace the operator $(-\partial_t)^{1-\beta}$ by $(\partial_t)^{1-\beta}$.

$\bullet$ If one asks instead the initial data $u_0$ to be in the smaller, but complete, space ${\eus{D}}_{A}(\theta,q)$ then one can drop assumptions \eqref{asmpt:homogeneousdomaindef} and \eqref{asmpt:homogeneousdomainintersect}, and the estimate \eqref{eq:BoundLqMaxReghomogeneous} still holds. However, one loose the possibility to compute the corresponding equivalent norm by the mean of real interpolation.

$\bullet$ The assumption \textit{(ii)} is probably not necessary for the case $\alpha<0$. However, it is not clear in this case how to prove the left hand side of the estimate \eqref{eq:BoundLqMaxReghomogeneous}. Indeed, our approach require to consider the action of $A^{1+\alpha}$, see \textbf{Step 3} in the proof.
\end{remark}

\begin{proof}[{of Theorem \ref{thm:LqMaxRegUMDHomogeneous}}]Let $q\in(1,+\infty)$, $\alpha\in(-1+1/q,1,q)$. Throughout this proof, and without loss of generality, we assume $T=+\infty$.

\textbf{Step 1:} The $\dot{\mathrm{H}}^{\alpha,q}$-maximal regularity estimate, for $u_0=0$, $f\in\dot{\mathrm{H}}^{\alpha,q}(\mathbb{R}_+,X)$. Mixed derivatives estimates.

We recall that $A$ has the $\mathrm{L}^q$-maximal regularity property.

Now, we use the fact $\partial_t$ and $A$ have their resolvent that commutes with each other, we have 
\begin{align*}
    (\partial_t+A)^{-1} = (\partial_t)^{-\alpha}(\partial_t+A)^{-1}(\partial_t)^{\alpha}\text{.}
\end{align*}
This equality and the isomorphism property of $(\partial_t)^{\alpha}$ \eqref{eq:isomDirTimeDeriv} yield
\begin{align*}
    \lVert (\partial_t u, Au)\rVert_{\dot{\mathrm{H}}^{\alpha,q}(\mathbb{R}_+,X)} \lesssim_{A,q,\alpha} \lVert f\rVert_{\dot{\mathrm{H}}^{\alpha,q}(\mathbb{R}_+,X)}\text{ . }
\end{align*}

And for the same reasons, from the $\mathrm{L}^q$-setting, if $A$ has BIP on $X$ of type $\theta_A<\frac{\pi}{2}$, by \cite[Proposition~2.4]{Pruss2002} for all $\beta\in[0,1]$, we have
\begin{align}\label{eq:mixedDerivEstimateDir}
    \lVert (\partial_t)^{1-\beta} A^{\beta} u\rVert_{\dot{\mathrm{H}}^{\alpha,q}(\mathbb{R}_+,X)} \lesssim_{A,q,\alpha} \lVert f\rVert_{\dot{\mathrm{H}}^{\alpha,q}(\mathbb{R}_+,X)}\text{. }
\end{align}
This estimate will be useful later.

Concerning the estimate \eqref{eq:mixedDerivEstimate} (with $u_0\in{\eus{D}}_{A}(1+\alpha-1/q,q)$), it suffices to assume that the right hand side of \eqref{eq:BoundLqMaxReghomogeneous} holds. Indeed, in this case it suffices to apply manually the three lines lemma, see e.g. the proof of \cite[Theorem~2.7]{bookLunardiInterpTheory}, to the holomorphic families (of operators) $(e^{(z-\beta)^2}(-\partial_t)^{1-z}A^{z}(\partial_{t}+A)^{-1})_{0\leqslant \Re (z) \leqslant 1}$ and $(e^{(z-\beta)^2}(-\partial_t)^{1-z}A^{z}[e^{-tA}(\cdot)])_{0\leqslant \Re (z) \leqslant 1}$, provided $\beta \in(0,1)$ is fixed. The proof of the boundedness is then carried over by \eqref{eq:BoundLqMaxReghomogeneous} and BIP of $A$ and of $-\partial_t$ respectively. Details are left to the reader.

\textbf{Step 2:} The trace estimate when $\alpha\in[0,1/q)$. Let $u_0\in \mathring{\eus{D}}_{A}(1+\alpha-{1}/{q},q)$, $f\in\dot{\mathrm{H}}^{\alpha,q}(\mathbb{R}_+,X)$.
The solution $u$ must be given for all $t>0$, by
\begin{align*}
    u(t) = e^{-t A}u_0 + \int_{0}^{t}e^{-sA}f(t-s)\,\mathrm{d}s \text{. }
\end{align*}
Corollary \ref{cor:ContinuityDAthetaq} tells us that
\begin{align*}
    u\in \mathrm{C}^0_b(\mathbb{R}_+,\mathring{\eus{D}}_{A}(1+\alpha-{1}/{q},q)))\text{,}
\end{align*}
with the estimate
\begin{align*}
    \lVert u \rVert_{\mathrm{L}^\infty(\mathbb{R}_+,\mathring{\eus{D}}_{A}(1+\alpha-{1}/{q},q))} \lesssim_{A,q,\alpha} \lVert f\rVert_{\dot{\mathrm{H}}^{\alpha,q}(\mathbb{R}_+,X)} + \lVert u_0\rVert_{\mathring{\eus{D}}_{A}(1+\alpha-{1}/{q},q)} \text{.}
\end{align*}
Since $f=\partial_t u+ Au$ and $e^{-tA} u_0 = u(t)- \int_{0}^{t}e^{-sA}f(t-s)\,\mathrm{d}s$, for all $t>0$, the triangle inequality leads to
\begin{align*}
    \lVert u \rVert_{\mathrm{L}^\infty(\mathbb{R}_+,\mathring{\eus{D}}_{A}(1+\alpha-{1}/{q},q))} \lesssim_{A,q,\alpha}& \lVert (\partial_t u, Au)\rVert_{\dot{\mathrm{H}}^{\alpha,q}(\mathbb{R}_+,X)} + \left(\int_{0}^{+\infty}\Bigg( \tau^{\frac{1}{q}-\alpha} \lVert Au(\tau)\rVert_X\Bigg)^q \frac{\mathrm{d}\tau}{\tau}\right)^{\frac{1}{q}}\\
    &\quad +  \left(\int_{0}^{+\infty}\Bigg( \tau^{\frac{1}{q}-\alpha} \left\lVert A \int_{0}^{\tau} e^{-sA}f(\tau-s)\,\mathrm{d}s  \right\rVert_X\Bigg)^q \frac{\mathrm{d}\tau}{\tau}\right)^{\frac{1}{q}} \text{ . }
\end{align*}
Thus, by the Hardy-Sobolev inequality, Lemma \ref{lem:HardySobolevInequality}, we obtain
\begin{align*}
    \lVert u \rVert_{\mathrm{L}^\infty(\mathbb{R}_+,\mathring{\eus{D}}_{A}(1+\alpha-{1}/{q},q))} \lesssim_{A,q,\alpha}& \lVert (\partial_t u, Au)\rVert_{\dot{\mathrm{H}}^{\alpha,q}(\mathbb{R}_+,X)} + \left\lVert \tau\mapsto  A\int_{0}^{\tau}e^{-(\tau-s)A}f(s)\,\mathrm{d}s\right\rVert_{\dot{\mathrm{H}}^{\alpha,q}(\mathbb{R}_+,X)}\text{.}
\end{align*}
Now, we may apply \textbf{Step 1} on the last term, by the triangle inequality, since, again, $f=\partial_t u + Au$, we deduce
\begin{align*}
    \lVert u \rVert_{\mathrm{L}^\infty(\mathbb{R}_+,\mathring{\eus{D}}_{A}(1+\alpha-{1}/{q},q))} \lesssim_{A,q,\alpha}& \lVert (\partial_t u, Au)\rVert_{\dot{\mathrm{H}}^{\alpha,q}(\mathbb{R}_+,X)} + \lVert f \rVert_{\dot{\mathrm{H}}^{\alpha,q}(\mathbb{R}_+,X)}\\
    \lesssim_{A,q,\alpha}& \lVert (\partial_t u, Au)\rVert_{\dot{\mathrm{H}}^{\alpha,q}(\mathbb{R}_+,X)}\text{.}
\end{align*}

\textbf{Step 3:} The trace estimate when $\alpha\in(-1+1/q,0)$. Let $f\in \dot{\mathrm{H}}^{\alpha,q}(\mathbb{R}_+,X)$.

By Lemma \ref{lem:FromHaqToLqloc}, we have
\begin{align*}
    t\mapsto \int_{0}^{t}e^{-(t-s)A}f(s)\,\mathrm{d}s \in \mathrm{C}^0(\mathbb{R}_+,X)\text{ .}
\end{align*}
However, for $t,\tau>0$,
\begin{align*}
    e^{-\tau A}\int_{0}^{t}e^{-(t-s)A}f(s)\,\mathrm{d}s = \int_{0}^{t+\tau}e^{-(\tau + t-s)A}f(s)\,\mathrm{d}s - \int_{0}^{\tau}e^{-(\tau-s)A}f(s+t)\,\mathrm{d}s\text{.}
\end{align*}
So that if we set $v=(\partial_t+A)^{-1}f$, $v_t=(\partial_t+A)^{-1}[f(\cdot+t)]$, we obtain for $t,\tau>0$
\begin{align*}
    e^{-\tau A}v(t) = v(t+\tau)-v_t(\tau) \text{.}
\end{align*}
Therefore, by analyticity of the semigroup $(e^{-\tau A})_{\tau>0}$, and the triangle inequality, we obtain for $t>0$,
\begin{align*}
    \lVert v(t) \rVert_{\mathring{\eus{D}}_{A}(1+\alpha-{1}/{q},q)} \lesssim_{\alpha,q,A}\lVert A^{1+\alpha} v(\cdot + t) \rVert_{\mathrm{L}^q(\mathbb{R}_+,X)} + \lVert A^{1+\alpha} v_t \rVert_{\mathrm{L}^q(\mathbb{R}_+,X)}\text{.}
\end{align*}
We can now apply \eqref{eq:mixedDerivEstimateDir} with $\beta = 1+\alpha$, and use the translation invariance of Sobolev norms, yielding
\begin{align*}
    \lVert v(t) \rVert_{\mathring{\eus{D}}_{A}(1+\alpha-{1}/{q},q)} \lesssim_{\alpha,q,A}\lVert f \rVert_{\dot{\mathrm{H}}^{\alpha,q}(\mathbb{R}_+,X)}\text{.}
\end{align*}
Now, for $u= e^{-(\cdot)A}u_{0}+v$, provided $u_0\in \mathring{\eus{D}}_{A}(1+\alpha-{1}/{q},q)$, we deduce
\begin{align*}
    \lVert u \rVert_{\mathrm{L}^\infty(\mathbb{R}_+,\mathring{\eus{D}}_{A}(1+\alpha-{1}/{q},q))} \lesssim_{\alpha,q,A}\lVert f \rVert_{\dot{\mathrm{H}}^{\alpha,q}(\mathbb{R}_+,X)} + \lVert u_0 \rVert_{\mathring{\eus{D}}_{A}(1+\alpha-{1}/{q},q)}\text{.}
\end{align*}
Again, to obtain the left hand side of \eqref{eq:BoundLqMaxReghomogeneous}, as in the previous \textbf{Step 2}, it suffices to estimate $u_0$ in $\mathring{\eus{D}}_{A}(1+\alpha-{1}/{q},q)$-norm. However, such estimate may involve the action of fractional powers of $A$ on $e^{-\tau A}u_0$, for which the meaning is not clear when $u_0\in \mathring{\eus{D}}_{A}(\theta,q)$. To circumvent this issue, we use the fact that $\mathrm{D}(A)$ is dense in $\mathring{\eus{D}}_{A}(1+\alpha-{1}/{q},q)$  by \cite[Lemma~2.10]{DanchinHieberMuchaTolk2020}. Thus, let $(u_{0,n})_{n\in\mathbb{N}}$ be a sequence in $\mathrm{D}(A)$ which converges to $u_0$ in $\mathring{\eus{D}}_{A}(1+\alpha-{1}/{q},q)$.
We set for all $n\in\mathbb{N}$, $u_n:= e^{-(\cdot)A}u_{0,n}+v$.

By analyticity of the semigroup $(e^{-\tau A})_{\tau >0}$, and by the identity
\begin{align*}
    (-\partial_{t})^{\alpha}[Au_n -Av](\tau) = (-\partial_{t})^{\alpha}[Ae^{-(\cdot) A}u_{0,n}](\tau) = A^{1+\alpha}e^{-\tau A}u_{0,n} \text{,}
\end{align*}
we are able to deduce that
\begin{align*}
    \lVert u_{0,n} \rVert_{\mathring{\eus{D}}_{A}(1+\alpha-{1}/{q},q)} &= \left(\int_{0}^{+\infty} \left(\tau^{\frac{1}{q}-\alpha}\lVert A e^{-\tau A} u_{0,n}\rVert_X\right)^q \frac{\mathrm{d}\tau}{\tau}\right)^{\frac{1}{q}}\\
    &\lesssim_{\alpha,q,A} \left(\int_{0}^{+\infty} \lVert A^{1+\alpha} e^{-\tau A} u_{0,n}\rVert_X^q {\mathrm{d}\tau}\right)^{\frac{1}{q}}\\
    &\lesssim_{\alpha,q,A} \left(\int_{0}^{+\infty} \lVert (-\partial_{t})^{\alpha}[Au_n -Av](\tau)\rVert_X^q {\mathrm{d}\tau}\right)^{\frac{1}{q}} \text{.}
\end{align*}
Finally, we can use the isomorphism property \eqref{eq:isomNeuTimeDeriv} and the triangle inequality to obtain
\begin{align*}
    \lVert u_{0,n} \rVert_{\mathring{\eus{D}}_{A}(1+\alpha-{1}/{q},q)} &\lesssim_{\alpha,q,A} \lVert A u_n \rVert_{\dot{\mathrm{H}}^{\alpha,q}(\mathbb{R}_+,X)} + \lVert A v \rVert_{\dot{\mathrm{H}}^{\alpha,q}(\mathbb{R}_+,X)}\\
    &\lesssim_{\alpha,q,A} \lVert A u_n \rVert_{\dot{\mathrm{H}}^{\alpha,q}(\mathbb{R}_+,X)} + \lVert f \rVert_{\dot{\mathrm{H}}^{\alpha,q}(\mathbb{R}_+,X)}\\
    &\lesssim_{\alpha,q,A}\lVert (\partial_t u_n, A u_n) \rVert_{\dot{\mathrm{H}}^{\alpha,q}(\mathbb{R}_+,X)}\text{ . }
\end{align*}
The proof ends here since one can pass to the limit as $n$ goes to infinity.

It remains to prove the continuity in time with values in $\mathring{\eus{D}}_{A}(1+\alpha-{1}/{q},q)$ which follows from a density argument\footnote{No need of completeness here, since the involved limits are already constructed.}.
\end{proof}

\typeout{}                                
\bibliographystyle{alpha}
{\footnotesize
\bibliography{Biblio}}

\end{document}